\UseRawInputEncoding
\documentclass[12pt,reqno]{amsart}
\usepackage{fullpage}

\newtheorem{theorem}{Theorem}[section]
\newtheorem{lemma}[theorem]{Lemma}

\newtheorem{cor}[theorem]{Corollary}

\newtheorem{prob}[theorem]{Problem}
\usepackage{graphicx}
\usepackage{color}
\usepackage[dvipsnames]{xcolor}
\usepackage{subfigure}
\usepackage{amssymb}
\usepackage{amsmath,mathrsfs}
\usepackage{colonequals}
\usepackage[backref=page]{hyperref}

\theoremstyle{definition}
\newtheorem{definition}[theorem]{Definition}

\newtheorem{remark}[theorem]{Remark}

\renewcommand{\subset}{\subseteq}

\renewcommand{\epsilon}{\varepsilon}
\renewcommand{\nu}{v}

\newcommand{\abs}[1]{\left|#1\right|}                   
\newcommand{\vnorm}[1]{\left\|#1\right\|}    
\newcommand{\vnormf}[1]{\|#1\|}                         
\newcommand{\vnormt}[1]{\left\|#1\right\|}    

\renewcommand{\d}{\mathrm{d}}

\newcommand{\R}{\mathbb{R}}

\newcommand{\embolden}[1]{\textbf {#1}}

\newcommand{\redA}{\partial^{*}\Omega}
\newcommand{\redb}{\partial^{*}}
\newcommand{\sdimn}{n}
\newcommand{\adimn}{n+1}
\newcommand{\scon}{\lambda}
\newcommand{\pcon}{\delta}
\newcommand{\smax}{\mathrm{max}_{\beta}}

\begin{document}

\title{Convex Cylinders and the Symmetric Gaussian Isoperimetric Problem}

\author{Steven Heilman}
\address{Department of Mathematics, University of Southern California, Los Angeles, CA 90089-2532}
\email{stevenmheilman@gmail.com}
\date{\today}
\thanks{Supported by NSF Grant CCF 1911216.}


\begin{abstract}
Let $\Omega$ be a measurable Euclidean set in $\mathbb{R}^{n}$ that is symmetric, i.e. $\Omega=-\Omega$, such that $\Omega\times\mathbb{R}$ has the smallest Gaussian surface area among all measurable symmetric sets of fixed Gaussian volume.  We conclude that either $\Omega$ or $\Omega^{c}$ is convex.  Moreover, except for the case $H(x)=\langle x,N(x)\rangle+\lambda$ with $H\geq0$ and $\lambda<0$, we show there exist a radius $r>0$ and an integer $0\leq k\leq n-1$ such that after applying a rotation, the boundary of $\Omega$ must satisfy $\partial\Omega= rS^{k}\times\mathbb{R}^{n-k-1}$, with $\sqrt{n-1}\leq r\leq\sqrt{n+1}$ when $k\geq1$.  Here $S^{k}$ denotes the unit sphere of $\mathbb{R}^{k+1}$ centered at the origin, and $n\geq1$ is an integer.  One might say this result nearly resolves the symmetric Gaussian conjecture of Barthe from 2001.
\end{abstract}
\maketitle
\setcounter{tocdepth}{1}
\tableofcontents
%
%
%

\section{Introduction}

In \cite{barthe01}, Barthe asked: which symmetric measurable Euclidean set of fixed Gaussian measure minimizes its Gaussian surface area?  A subset $\Omega\subset\R^{\adimn}$ is symmetric if $\Omega=-\Omega$.  It is well-known that a half space minimizes Gaussian surface area, among Euclidean sets of fixed Gaussian measure \cite{borell75,sudakov74}.  Despite many proofs of this so-called Gaussian isoperimetric inequality \cite{borell85,ledoux94,ledoux96,bobkov97,burchard01,borell03,mossel15,mossel12,eldan13,mcgonagle15,barchiesi16}, additionally restricting to symmetric sets causes additional difficulties.  Indeed, not much progress seemed to be made on Barthe's question from 2001 until \cite{heilman17}, where various candidate optimal sets were ruled out, following \cite{lamanna17}, where the local minimality and non-minimality of balls centered at the origin was demonstrated.  Also, as suggested by Morgan \cite[Conjecture 1.14]{heilman17}, the answer to Barthe's question should depend on the Gaussian volume restriction.  For example, among symmetric Euclidean sets with Gaussian measure very close to $1$ or $0$, a one-dimensional slab (or its complement) should minimize Gaussian surface area.  This fact was demonstrated in \cite{barchiesi18}, proving the first known case of Barthe's question.  The proof of \cite{barchiesi18} considers minimizing the Gaussian surface area plus a ``penalty term,'' penalizing sets that have a nonzero Gaussian center of mass.  Considerable effort is then required to show that minimizers of this functional are symmetric, and they are slabs (or complements of slabs), for an appropriately chosen constant in front of the penalty term.

In this paper, we demonstrate that convex cylinders centered at the origin are the only sets that are stable for the Gaussian surface area, after taking a product with $\R$.  Our approach is rather different than that of \cite{barchiesi18}.  We begin by adapting the Colding-Minicozzi theory of entropy from \cite{colding12a}.  In \cite{colding12a}, it is shown that for self-shrinkers, i.e. surfaces $\Sigma$ satisfying $H(x)-\langle x,N(x)\rangle=0$ for all $x\in\Sigma$, $H$ is an eigenfunction of the second variation operator $L$.  (Here $H(x)$ denotes the mean curvature of $\Sigma$ at $x$, i.e. the divergence of the exterior pointing unit normal vector $N(x)$ at $x$.  Also $L$ is defined in \eqref{three4.5}.)  Namely, $LH=2H$ (see \eqref{three9}).  This explicit eigenfunction with eigenvalue $2$ crucially allows a stability analysis of self-shrinkers with respect to the Colding-Minicozzi entropy functional.  The Colding-Minicozzi entropy of a surface $\Sigma$ is the supremum over translations and dilations of the Gaussian surface area.  Critical points of this functional are self-shrinkers \cite{colding12a}, hence the focus of \cite{colding12a} on these surfaces.

\subsection{Adapting the Colding-Minicozzi Theory}

In this paper, we are concerned with a stability analysis of the Gaussian surface area.  Critical points of the Gaussian surface area functional satisfy the more general condition that there exists some $\scon\in\R$ such that $H(x)-\langle x,N(x)\rangle=\scon$ for all $x\in\Sigma$.  In this setting, the mean curvature $H$ is no longer an eigenfunction of the second variation operator $L$ (it is an almost eigenfunction though; see \eqref{three9}.)  So, one cannot directly use the stability analysis of \cite{colding12a} for the Gaussian surface area itself.  However, appropriate modifications of the arguments of \cite{colding12a} can be made successful in certain cases.

Our general strategy is to show that the second variation operator has an eigenvalue larger than $2$.  Finding such an eigenvalue is insufficient to classify symmetric sets with minimal Gaussian surface area.  However, if we take a product of an eigenfunction multiplied by a function $x_{\adimn}^{2}-1$ in an orthogonal direction, as an input to the second variation formula, then we obtain a Gaussian volume preserving and Gaussian surface area decreasing perturbation of $\Omega\times\R$ (see Lemma \ref{orthlem}).  So, the main task is to look for (approximate) eigenfunctions of $L$ with eigenvalue larger than $2$.  To accomplish this task, it seems necessary to split into a few different cases.

In \cite{colding12a}, the stability of the entropy is split into two cases, according to whether or not the surface $\Sigma$ is mean convex (i.e. if $H$ changes sign on $\Sigma$).  In the case that $H$ changes sign on $\Sigma$, since $H$ itself is an eigenfunction of $L$ with eigenvalue $2$, another eigenfunction of $L$ with a larger eigenvalue must exist.  In our more general setting, since $H$ is no longer an eigenfunction of $L$, the observation that there exists another eigenfunction with a larger eigenvalue of \cite{colding12a} no longer applies.  So, we instead use a product of $\max(H,0)$ multiplied by a function $x_{\adimn}^{2}-1$ in an orthogonal direction, as an input to the second variation formula.

In the case that $H$ does not change sign and $H(x)-\langle x,N(x)\rangle=0$ on $\Sigma$, a curvature bound of the second fundamental form is proven in \cite{colding12a} that allows other eigenfunctions of $L$ to be used in the second variation formula, and no curvature bound needs to be proven.  In our more general setting where $H(x)-\langle x,N(x)\rangle=\scon$, this curvature bound seems difficult to prove in general, so that some functions might not be usable in the second variation formula.  Fortunately, we can split into two sub-cases.  In the case that $H$ and $\scon$ have the same sign, we can use $H$ itself in the second variation formula.  In the case that $H$ and $\scon$ have opposite signs, the curvature bound of \cite{colding12a} can be proven, allowing us to use a function of the second fundamental form $A$ in the second variation formula.  However, in this case, the second variation of $H$ itself is not helpful, so we needed to come up with another function to input into the second variation formula, namely the largest eigenvalue of $A$.

It turns out that the largest eigenvalue $\alpha$ of the second fundamental form $A$ is an almost eigenfunction of $L$ (see \eqref{three9p}).  (To avoid issues with differentiability and integration by parts, we actually use a smoothed version of the largest eigenvalue of $A$; see Lemma \ref{softeig}.)  So, $\max(\alpha,0)$ can be used in our second variation formula as long as $\alpha>0$ on a set of positive measure on $\Sigma$.  If no such set exists where $\alpha$ is positive, then all eigenvalues of $A$ are negative everywhere, so that $\Sigma$ is the boundary of a convex set.

In the case that the mean curvature $H$ does not change sign, there are two sub-cases to consider.  If $H\geq0$ and $\lambda>0$, then we can either use $H$ itself in the second variation formula, or use a Huisken-type classification from \cite{heilman17} (which does not use any second variation computations).  However, if $H\geq0$ and $\lambda<0$, then we are only able to deduce convexity of $\Omega$ or $\Omega^{c}$.  In fact, no Huisken-type classification can occur in this case.  We will discuss this case further in Section \ref{badcase}.

There is still one final case we have not mentioned, namely $H(x)=\langle x,N(x)\rangle=0$ on $\Sigma$, i.e. that $\Sigma$ is a Gaussian minimal cone.  In this last case, these sets cannot minimize Barthe's problem.  This follows by adapting an argument of \cite{zhu16}, which itself adapted an argument of Simons \cite{simons68}.  We use a perturbation of the surface that is a product of a radial and angular component.  The radial component is chosen to preserve Gaussian volume, and the angular component is chosen to have a large eigenvalue of the second variation operator.

The Colding-Minicozzi theory \cite{colding12a,colding12} was originally designed to use the Gaussian surface area to investigate singularities of mean-curvature flows.  A connection of Gaussian surface area to mean-curvature flow was established by Huisken \cite{huisken90,huisken93}, and \cite{colding12a} greatly extended this connection.  As a continuation of \cite{heilman17}, this paper instead applies the Colding-Minicozzi theory to a Gaussian isoperimetric conjecture.

\subsection{The case $H\geq0$ and $\lambda<0$}\label{badcase}

As mentioned above, in the case $H\geq0$ and $\lambda<0$, we can only conclude that $\Omega$ or $\Omega^{c}$ is convex, unlike in other cases where we can show that $\partial\Omega$ is a round cylinder centered at the origin.  The compact version of this convexity statement was proven in \cite{lee22}, though compactness was crucially used there, so it is unclear if the argument there generalizes to the noncompact case.  Part of the difficulty of this case is that Huisken's classification no longer holds \cite{huisken90,huisken93}.  Indeed, it is known that, for every integer $m\geq2$, there exists $\scon=\scon_{m}<0$ and there exists a convex embedded curve $\Gamma_{m}\subset\R^{2}$ satisfying $H(x)=\langle x,N(x)\rangle+\lambda$ as in Lemma \ref{varlem}, and such that $\Gamma_{m}$ is symmetric with respect to a rotation by an angle $2\pi/m$ (and $\Gamma_{m_{1}}\neq\Gamma_{m_{2}}$ if $m_{1}\neq m_{2}$, and also $\Gamma_{m}$ is not symmetric with respect to a rotation by an angle smaller than $2\pi/m$)  \cite[Theorem 1.2, Theorem 1.3, Proposition 3.2]{chang17}.  Consequently, $\Gamma_{m}\times\R^{\sdimn-2}\subset\R^{\adimn}$ also satisfies $H(x)=\langle x,N(x)\rangle+\lambda$.  So, Huisken's classification cannot possibly hold, at least when $\scon<0$.

In the case of even $m\geq6$, the curve $\Gamma_{m}$ can be shown to be unstable by \cite[Corollary 11.9]{heilman17}, since there are more than four nodal domains corresponding to an infinitesimal rotation of $\Gamma_{m}$.  However, it is unclear how to prove instability for the cases $m=2$ and $m=4$.  As shown at the end of \cite{chang17}, there appears to be an infinite family of curves with a single symmetry by a rotation of an angle $\pi$, and it is unclear if any of our arguments can show these curves are unstable for the Gaussian perimeter.

\subsection{Future Directions and Remaining Cases of Barthe's Problem}

The ultimate goal of Barthe's Problem \ref{prob1} is to find the minimum Gaussian perimeter of a set in $\R^{\adimn}$, where we take the infimum over all dimensions $\sdimn\geq0$.  That is, we would like to determine the isoperimetric profile $I_{\infty}$ depicted in Figure \ref{fig1}.

Theorem \ref{thm1} classifies those sets of the form $\Omega\times\R$ that are stable for the Gaussian surface area.  If a set $\Omega$ minimizes Problem \ref{prob1} and its surface area is achieved in the definition of $I_{\infty}$, then the set $\Omega\times\R$ must also be stable.  However, a priori, it could occur that for each $n\geq0$, there exists a set $\Omega_{n}\subset\R^{\adimn}$ that minimizes Problem \ref{prob1} for a measure constraint $\gamma_{\adimn}(\Omega_{n})=c\in(0,1)$, but such that $\Omega_{n}\times\R$ is unstable.  Put another way, it could occur that a value $I_{\infty}(c)$ is not achieved for any set of finite dimension.  In such a case, Theorem \ref{thm1} does not say anything about Problem \ref{prob1}.  And in fact, we do expect this to happen, but only when $c=1/2$.  It is conjectured that, for any $c\neq1/2$, $I_{\infty}(c)$ is achieved by a set of finite dimension.  And in the case $c=1/2$, $I_{\infty}(c)$ is achieved by a limit of spheres $S^{n}$ of radius approximately $\sqrt{n}$.  With this picture in mind, Theorem \ref{thm1} should cover many cases of Problem \ref{prob1} (except when $c=1/2$).  It was conjectured by Morgan that, if $\Omega\subset\R^{\adimn}$ minimizes Problem \ref{prob1}, there exists $r>0$ and there exists $0\leq k\leq n$ such that
$$\partial\Omega= rS^{k}\times\R^{n-k}.$$
We additionally conjecture that $r$ satisfies $\sqrt{n}\leq r\leq\sqrt{n+2}$ when $k\geq1$.

\begin{figure}
  \centering
  \includegraphics[width=.7\textwidth]{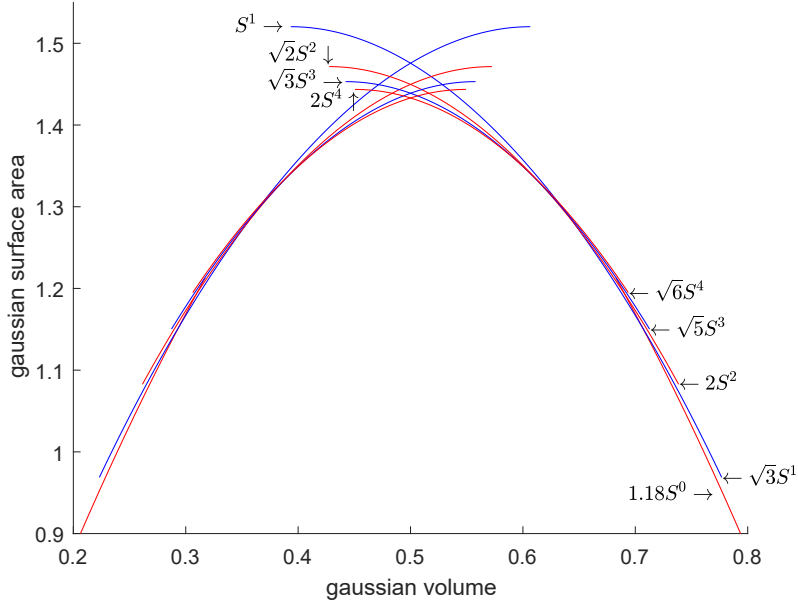}
\caption{Gaussian surface area of $B(0,s)\subset\R$ and of $B(0,r)\subset\R^{\adimn}$, where $\sqrt{n}\leq r\leq\sqrt{n+2}$ for each $1\leq n\leq 4$, together with their complements.}
\label{fig2}
\end{figure}

\section{Statement of Results}\label{secres}
We begin by defining the Gaussian density:
$$\gamma_{\sdimn}(x)\colonequals (2\pi)^{-\sdimn/2}e^{-\vnormt{x}^{2}/2},\qquad
\vnormt{x}^{2}\colonequals\sum_{i=1}^{m}x_{i}^{2},\qquad
\forall\,x=(x_{1},\ldots,x_{m})\in\R^{m}.
$$
We also denote $\gamma_{\adimn}(\cdot)$ as a measure of Lebesgue measurable sets in $\R^{\adimn}$.  For a set $\Sigma\subset\R^{\adimn}$ with Hausdorff dimension $\sdimn$, we denote its Gaussian surface area as
$$\int_{\Sigma}\gamma_{\sdimn}(x)\,\d x\colonequals\liminf_{\epsilon\to0^{+}}
\frac{1}{2\epsilon}\int_{\{x\in\R^{\adimn}\colon\exists\,y\in\Sigma,\,\vnormt{x-y}<\epsilon\}}\gamma_{\sdimn}(x)\,\d x.$$

Our main problem of interest is the following.
\begin{prob}[\embolden{Symmetric Gaussian Problem}, \cite{barthe01}]\label{prob1}
Fix $0<c<1$.  Minimize $$\int_{\partial \Omega}\gamma_{\sdimn}(x)\,\d x$$ over all (measurable) subsets $\Omega\subset\R^{\adimn}$ satisfying $\Omega=-\Omega$ and $\gamma_{\adimn}(\Omega)=c$.
\end{prob}
Unless otherwise stated, all sets discussed in this paper will be Lebesgue measurable.
\begin{remark}\label{crk}
If $\Omega$ minimizes Problem \ref{prob1}, then $\Omega^{c}$ also minimizes Problem \ref{prob1}, with $c$ replaced by $1-c$.
\end{remark}
\begin{definition}
For any integer $k\geq0$, define the $k$-dimensional sphere of radius one centered at the origin to be
$$S^{k}\colonequals\{x\in\R^{k+1}\colon\vnorm{x}=1\}.$$
\end{definition}

\begin{figure}[ht!]
\centering
\def\svgwidth{.5\textwidth}
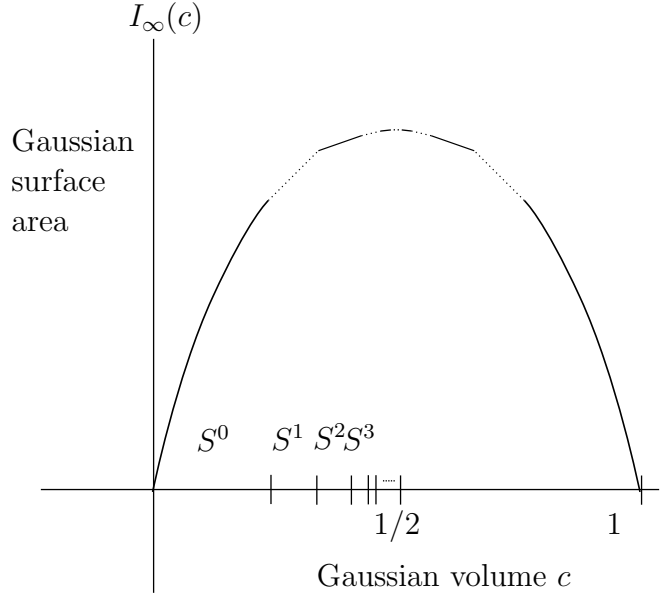
\caption{The conjectured form of the symmetric Gaussian isoperimetric profile.}
$$I_{\infty}(c)\colonequals\inf_{n\geq0}\inf_{\substack{\Omega\subset\R^{\adimn}\\ \Omega=-\Omega\\ \gamma_{\adimn}(\Omega)=c}}\int_{\partial\Omega}\gamma_{\sdimn}(x)\,\d x$$
\label{fig1}
\end{figure}

\begin{theorem}[\embolden{Main Theorem}]\label{thm1}
Let $\Omega\subset\R^{\adimn}$ be a measurable set.  Assume that $\Omega\times\R$ minimizes minimizing Problem \ref{prob1}.  Then $\Omega$ or $\Omega^{c}$ is convex.

Moreover, unless $H\geq0$ and $\lambda<0$ (or $H\leq 0$ and $\lambda>0$), $\exists$ $r>0$, $\exists$ $0\leq k\leq \sdimn$ such that
$$\partial\Omega=r S^{k}\times\R^{\sdimn-k},$$
after rotating $\Omega$ if necessary.  And if $k\geq1$, then $\sqrt{\sdimn}\leq r\leq\sqrt{\sdimn+2}$.
\end{theorem}
Theorem \ref{thm1} will be split into several cases, resulting from the combination of Theorems \ref{thm3}, \ref{thm5}, \ref{thm6} and \ref{cor7}.

It is shown in \cite{lamanna17} that the ball centered at the origin of $\R^{\adimn}$ is a local minimum of Problem \ref{prob1}, when the ball's radius $r$ satisfies $r<\sqrt{\sdimn+2}$.  (Consequently, the complement of a ball centered at the origin is a local minimum of Problem \ref{prob1}, when the ball's radius $r$ satisfies $r<\sqrt{\sdimn+2}$, by Remark \ref{crk}.)  It is also shown in \cite{lamanna17} that a ball of radius $r>\sqrt{\sdimn+2}$ is not a local minimum of Problem \ref{prob1}.  Our results (e.g. Case 2 of Theorem \ref{cor7}) show, for a ball $B(0,r)\subset\R^{\adimn}$ of radius $r$ centered at the origin, $B(0,r)\times\R$ does not minimize problem \ref{prob1} when $r\leq\sqrt{\sdimn}$.  So, if $n\geq1$, the Gaussian surface area of $B(0,r)\subset\R^{\adimn}$ should only appear as a value in Figure \ref{fig1} when $\sqrt{n}\leq r\leq\sqrt{n+2}$.

\begin{remark}
As shown e.g. in the introduction of \cite{heilman17}, $\sqrt{n+1}S^{n}$ has asymptotic Gaussian surface area $\sqrt{2}$.
$$\lim_{n\to\infty}\int_{\{x\in\R^{\adimn}\colon \vnorm{x}=\sqrt{\adimn}\}}\gamma_{\sdimn}(x)\,\d x=\sqrt{2}.$$
In contrast, a half space with Gaussian measure $1/2$ has Gaussian surface area $\gamma_{0}(0)=1$.
\end{remark}
\begin{remark}
For any $k\geq1$, denote $c_{k}\colonequals\int_{\{x\in\R^{k+1}\colon\vnorm{x}=\sqrt{k}\}}\gamma_{k}(x)\,\d x$.  As shown in \cite[Lemma A.4]{stone94}, this sequence is decreasing: $c_{1}>c_{2}>c_{3}>\cdots$.
\end{remark}

\section{Preliminaries}\label{secpre}

We say that $\Sigma\subset\R^{\adimn}$ is an $\sdimn$-dimensional $C^{\infty}$ manifold if $\Sigma$ can be locally written as the graph of a $C^{\infty}$ function.  For any $(\adimn)$-dimensional $C^{\infty}$ manifold $\Omega\subset\R^{\adimn}$ with boundary, we denote
\begin{equation}\label{c0def}
\begin{aligned}
C_{0}^{\infty}(\Omega;\R^{\adimn})
&\colonequals\{f\colon \Omega\to\R^{\adimn}\colon f\in C^{\infty}(\Omega;\R^{\adimn}),\, f(\partial \Omega)=0,\\
&\qquad\qquad\qquad\exists\,r>0,\,f(\Omega\cap(B(0,r))^{c})=0\}.
\end{aligned}
\end{equation}
We also denote $C_{0}^{\infty}(\Omega)\colonequals C_{0}^{\infty}(\Omega;\R)$.  We let $\mathrm{div}$ denote the divergence of a vector field in $\R^{\adimn}$.  For any $r>0$ and for any $x\in\R^{\adimn}$, we let $B(x,r)\colonequals\{y\in\R^{\adimn}\colon\vnormt{x-y}\leq r\}$ be the closed Euclidean ball of radius $r$ centered at $x\in\R^{\adimn}$.

\begin{definition}[\embolden{Reduced Boundary}]
A measurable set $\Omega\subset\R^{\adimn}$ has \embolden{locally finite surface area} if, for any $r>0$,
$$\sup\left\{\int_{\Omega}\mathrm{div}(X(x))\,\d x\colon X\in C_{0}^{\infty}(B(0,r),\R^{\adimn}),\, \sup_{x\in\R^{\adimn}}\vnormt{X(x)}\leq1\right\}<\infty.$$
Equivalently, $\Omega$ has locally finite surface area if $\nabla 1_{\Omega}$ is a vector-valued Radon measure such that, for any $x\in\R^{\adimn}$, the total variation
$$
\vnormt{\nabla 1_{\Omega}}(B(x,1))
\colonequals\sup_{\substack{\mathrm{partitions}\\ C_{1},\ldots,C_{m}\,\mathrm{of}\,B(x,1) \\ m\geq1}}\sum_{i=1}^{m}\vnormt{\nabla 1_{\Omega}(C_{i})}
$$
is finite \cite{cicalese12}.

If $\Omega\subset\R^{\adimn}$ has locally finite surface area, we define the \embolden{reduced boundary} $\redA$ of $\Omega$ to be the set of points $x\in\R^{\adimn}$ such that
$$N(x)\colonequals-\lim_{r\to0^{+}}\frac{\nabla 1_{\Omega}(B(x,r))}{\vnormt{\nabla 1_{\Omega}}(B(x,r))}$$
exists, and it is exactly one element of $S^{\sdimn}$.
\end{definition}
For more background on the reduced boundary and its regularity, we refer to the discussion in Section 2 of \cite{barchiesi16}, \cite{ambrosio00} and \cite{maggi12}.  The following argument is essentially identical to \cite[Proposition 1]{barchiesi16}, so we omit the proof.

\begin{lemma}[\embolden{Existence}]\label{lemma51p}
There exists a set $\Omega\subset\R^{\adimn}$ minimizing Problem \ref{prob1}.
\end{lemma}

\subsection{Submanifold Curvature}\label{seccurvature}

Here we cover some basic definitions from differential geometry of submanifolds of Euclidean space.

Let $\nabla$ denote the standard Euclidean connection, so that if $X,Y\in C_{0}^{\infty}(\R^{\adimn},\R^{\adimn})$, if $Y=(Y_{1},\ldots,Y_{\adimn})$, and if $u_{1},\ldots,u_{\adimn}$ is the standard basis of $\R^{\adimn}$, then $\nabla_{X}Y\colonequals\sum_{i=1}^{\adimn}(X (Y_{i}))u_{i}$.  Let $N$ be the outward pointing unit normal vector of an $\sdimn$-dimensional hypersurface $\Sigma\subset\R^{\adimn}$.  For any vector $x\in\Sigma$, we write $x=x^{T}+x^{N}$, so that $x^{N}\colonequals\langle x,N\rangle N$ is the normal component of $x$, and $x^{T}$ is the tangential component of $x\in\Sigma$.  We let $\nabla^{\Sigma}\colonequals(\nabla)^{T}$ denote the tangential component of the Euclidean connection.

Let $e_{1},\ldots,e_{\sdimn}$ be an orthonormal frame of $\Sigma\subset\R^{\adimn}$.  That is, for a fixed $x\in\Sigma$, there exists a neighborhood $U$ of $x$ such that $e_{1},\ldots,e_{\sdimn}$ is an orthonormal basis for the tangent space of $\Sigma$, for every point in $U$ \cite[Proposition 11.17]{lee03}.

Define the \embolden{mean curvature}
\begin{equation}\label{three0.5}
H\colonequals\mathrm{div}(N)=\sum_{i=1}^{\sdimn}\langle\nabla_{e_{i}}N,e_{i}\rangle.
\end{equation}

Define the \embolden{second fundamental form} $A=(a_{ij})_{1\leq i,j\leq\sdimn}$ so that
\begin{equation}\label{three1}
a_{ij}=\langle\nabla_{e_{i}}e_{j},N\rangle,\qquad\forall\,1\leq i,j\leq \sdimn.
\end{equation}
Compatibility of the Riemannian metric says $a_{ij}=\langle\nabla_{e_{i}}e_{j},N\rangle=-\langle e_{j},\nabla_{e_{i}}N\rangle+ e_{i}\langle N,e_{j}\rangle=-\langle e_{j},\nabla_{e_{i}}N\rangle$, $\forall$ $1\leq i,j\leq \sdimn$.  So, multiplying by $e_{j}$ and summing this equality over $j$ gives
\begin{equation}\label{three2}
\nabla_{e_{i}}N=-\sum_{j=1}^{\sdimn}a_{ij}e_{j},\qquad\forall\,1\leq i\leq \sdimn.
\end{equation}


Using $\langle\nabla_{N}N,N\rangle=0$,
\begin{equation}\label{three4}
H\stackrel{\eqref{three0.5}}{=}\sum_{i=1}^{\sdimn}\langle \nabla_{e_{i}} N,e_{i}\rangle
\stackrel{\eqref{three2}}{=}-\sum_{i=1}^{\sdimn}a_{ii}.
\end{equation}

\subsection{First and Second Variation}

We will apply the calculus of variations to solve Problem \ref{prob1}. Here we present the rudiments of the calculus of variations.

The results of this section are well known to experts in the calculus of variations, and many of these results were re-proven in \cite{barchiesi16}.

Let $\Omega\subset\R^{\adimn}$ be an $(\adimn)$-dimensional $C^{2}$ submanifold with reduced boundary $\Sigma\colonequals\redA$.  Let $N\colon\redA\to S^{\sdimn}$ denote the unit exterior normal to $\redA$.  Let $X\colon\R^{\adimn}\to\R^{\adimn}$ be a vector field.  Unless otherwise stated, we assume that $X(x)$ is parallel to $N(x)$ for all $x\in\redA$, i.e.
\begin{equation}\label{nine2.4}
X(x)=\langle X(x),N(x)\rangle N(x),\qquad\forall\, x\in\redA.
\end{equation}
Let $\mathrm{div}$ denote the divergence of a vector field.  We write $X$ in its components as $X=(X_{1},\ldots,X_{\adimn})$, so that $\mathrm{div}X=\sum_{i=1}^{\adimn}\frac{\partial}{\partial x_{i}}X_{i}$.  Let $\Psi\colon\R^{\adimn}\times(-1,1)\to\R^{\adimn}$ such that
\begin{equation}\label{nine2.3}
\Psi(x,0)=x,\qquad\qquad\frac{\d}{\d s}\Psi(x,s)=X(\Psi(x,s)),\quad\forall\,x\in\R^{\adimn},\,s\in(-1,1).
\end{equation}
For any $s\in(-1,1)$, let $\Omega^{(s)}\colonequals\Psi(\Omega,s)$.  Note that $\Omega_{0}=\Omega$.  Let $\Sigma^{(s)}\colonequals\redb\Omega^{(s)}$.
\begin{definition}
We call $\{\Omega^{(s)}\}_{s\in(-1,1)}$ as defined above a \embolden{normal variation} of $\Omega\subset\R^{\adimn}$.  We also call $\{\Sigma^{(s)}\}_{s\in(-1,1)}$ a \embolden{normal variation} of $\Sigma=\partial\Omega$.
\end{definition}

\begin{lemma}[\embolden{First Variation}, {\cite[Lemma 2.4]{heilman17}}]\label{lemma10}  Let $X\in C_{0}^{\infty}(\R^{\adimn},\R^{\adimn})$.  Let $f(x)=\langle X(x),N(x)\rangle$ for any $x\in\redA$.  Then
\begin{equation}\label{nine1}
\frac{\d}{\d s}\Big|_{s=0}\gamma_{\adimn}(\Omega^{(s)})=\int_{\redA}f(x) \gamma_{\adimn}(x)\,\d x.
\end{equation}
\begin{equation}\label{nine2}
\frac{\d}{\d s}\Big|_{s=0}\int_{\redb \Omega^{(s)}}\gamma_{\sdimn}(x)\,\d x
=\int_{\redA}(H(x)-\langle N(x),x\rangle)f(x)\gamma_{\sdimn}(x)\,\d x.
\end{equation}
\end{lemma}
%
%

\section{Variations and Regularity}

In this section, we show that a minimizer of Problem \ref{prob1} exists, and the boundary of the minimizer is $C^{\infty}$ except on a set of Hausdorff dimension at most $n-7$.


The results below are repeated from \cite{heilman17}.

Unless otherwise stated, all sets $\Omega\subset\R^{\adimn}$ below are assumed to be measurable sets of locally finite surface area, such that the Gaussian surface area of $\Omega$, $\int_{\redA}\gamma_{\sdimn}(x)\,\d x$ is finite.

\begin{lemma}[\embolden{First Variation for Minimizers}, {\cite[Lemma 3.1]{heilman17}}]\label{varlem}
Let $\Omega\subset\R^{\adimn}$ minimize Problem \ref{prob1}.  Let $\Sigma\colonequals\redA$.  Then there exists $\scon\in\R$ such that, for any $x\in\redA$, $H(x)-\langle x,N(x)\rangle=\scon$.
\end{lemma}

\begin{lemma}[\embolden{Second Variation for Minimizers}, {\cite[Lemma 2.5, Lemma 3.2]{heilman17}}]\label{varlem2}
Let $\Omega\subset\R^{\adimn}$ minimize Problem \ref{prob1}.  Let $\Sigma\colonequals\redA$.  Then, for any $f\in C_{0}^{\infty}(\Sigma)$ such that $\int_{\Sigma}f(x)\gamma_{\sdimn}(x)\,\d x=0$, and such that $f(x)=f(-x)$ for all $x\in\Sigma$, there exists a vector field $X\in C_{0}^{\infty}(\R^{\adimn},\R^{\adimn})$ with $f(x)=\langle X(x),N(x)\rangle$ for all $x\in\Sigma$ such that $X(x)=X(-x)$ for all $x\in\R^{\adimn}$, and such that the variation of $\{\Omega^{(s)}\}_{s\in(-1,1)}$ corresponding to $X$ satisfies
$$\frac{\d^{2}}{\d s^{2}}\Big|_{s=0}\int_{\redb \Omega^{(s)}}\gamma_{\sdimn}(x)\,\d x=\int_{\Sigma}f(x)Lf(x)\gamma_{\sdimn}(x)\,\d x\leq0.$$
\end{lemma}%

\begin{lemma}[\embolden{Existence and Regularity}, {\cite[Lemma 3.3]{heilman17}}]\label{lemma51}
The minimum value of Problem \ref{prob1} exists.  That is, there exists a measurable set $\Omega\subset\R^{\adimn}$ 
such that $\Omega$ achieves the minimum value of Problem \ref{prob1}.  Also, $\redA$ is a $C^{\infty}$ manifold.  Moreover, if $\sdimn<7$, then $\partial\Omega\setminus\redA=\emptyset$, and if $n\geq7$, then the Hausdorff dimension of $\partial\Omega\setminus\redA$ is at most $n-7$.
\end{lemma}%

\section{Eigenfunctions of L}\label{seceig}

Let $e_{1},\ldots,e_{\sdimn}$ be an orthonormal frame for an orientable $\sdimn$-dimensional hypersurface $\Sigma\subset\R^{\adimn}$ with $\partial\Sigma=\emptyset$.  Let $\Delta\colonequals\sum_{i=1}^{\sdimn}\nabla_{e_{i}}\nabla_{e_{i}}$ be the Laplacian associated to $\Sigma$. Let $\nabla\colonequals\sum_{i=1}^{\sdimn}e_{i}\nabla_{e_{i}}$ be the gradient associated to $\Sigma$.  (The symbol $\nabla_{\cdot}(\cdot)$ still denotes the Euclidean connection, and the meaning of the symbol $\nabla$ should be clear from context.)  For any $\sdimn\times\sdimn$ matrix $B=(b_{ij})_{1\leq i,j\leq\sdimn}$, define $\vnormt{B}^{2}\colonequals\sum_{i,j=1}^{\sdimn}b_{ij}^{2}$.  For any $f\in C^{\infty}(\Sigma)$, define
\begin{equation}\label{three4.3}
\mathcal{L}f\colonequals \Delta f-\langle x,\nabla f\rangle.
\end{equation}
\begin{equation}\label{three4.5}
L f\colonequals \Delta f-\langle x,\nabla f\rangle+f+\vnormt{A}^{2}f.
\end{equation}
Note that there is a factor of $2$ difference between our definition of $L$ and the definition of $L$ in \cite{colding12a}.  Below we often remove the $x$ arguments of the functions for brevity.  We extend $L$ to matrices so that $(LB)_{ij}\colonequals L(B_{ij})$ for all $1\leq i,j\leq\sdimn$, and we can similarly extend $\nabla$ to matrices.
\begin{remark}\label{rk20}
Let $f,g\in C^{\infty}(\Sigma)$.  Using \eqref{three4.5}, we get the following product rule for $L$.
\begin{flalign*}
L(fg)
&=f\Delta g+g\Delta f+2\langle \nabla f,\nabla g\rangle-f\langle x,\nabla g\rangle-g\langle x,\nabla f\rangle+\vnormt{A}^{2}fg+fg\\
&=fLg+g\mathcal{L}f+2\langle \nabla f,\nabla g\rangle.
\end{flalign*}
\end{remark}

\begin{lemma}[\embolden{$H$ is almost an eigenfunction of $L$} {\cite[Proposition 1.2]{colding15}} {\cite[Lemma 2.1]{guang15}}, {\cite[Lemma 4.1]{heilman17}}]\label{lemma30}
Let $\Sigma\subset\R^{\adimn}$ be an orientable hypersurface.  Let $\scon\in\R$.  If
\begin{equation}\label{three0}
H(x)=\langle x,N(x)\rangle+\scon,\qquad\forall\,x\in\Sigma.
\end{equation}
Then
\begin{equation}\label{three9p}
LA=2A-\scon A^{2}.
\end{equation}
\begin{equation}\label{three9}
LH=2H+\scon\vnormt{A}^{2}.
\end{equation}
\end{lemma}

Equation \eqref{three9p} suggests that an eigenvalue $\alpha$ of $A$ should satisfy an equation similar to \eqref{three9p}.  That is, if we choose an orthonormal frame such that $A$ is diagonal at a point, then it should be the case that $L\alpha$ is nearly equal to $2\alpha-\lambda\alpha^{2}$.  (Near an umbilic point, the largest eigenvalue of $A$ might change quickly, so we should not expect $L\alpha$ to be equal to $2\alpha-\lambda\alpha^{2}$.)  Such a computation was done in \cite{lee22}.  Also, there is a technical issue, that an eigenvalue of $A$ might not be differentiable, or more importantly, we might not be able to integrate by parts using $\alpha$, i.e. we cannot use $\alpha$ in the second variation formula, Lemma \ref{varlem2}.

For these technical reasons, we instead use a softmax function in \eqref{smaxdef} to approximate the largest eigenvalue of $A$.  The function in \eqref{smaxdef} will be smooth, so we can integrate by parts with it and use it in Lemma \ref{varlem2}.  Moreover, since the function \eqref{smaxdef} does not rely on the choice of orthonormal frame, we can substitute \eqref{three9p} into the differentiation formula for \eqref{smaxdef} in \eqref{lsofeq} below.  There is an extra term that appears in \eqref{lsofeq} that results from the smoothing in \eqref{smaxdef}, but the sign of this extra term is determined, so it does not affect our calculations later very much.

Let $\beta>0$.  Define the ($\beta$-smoothed) maximum eigenvalue of $A$ by
\begin{equation}\label{smaxdef}
\smax(A)\colonequals\frac{1}{\beta}\log\Big(\mathrm{Tr}\,e^{\beta A}\Big).
\end{equation}
Here $\mathrm{Tr}$ denotes the trace of a square matrix.
\begin{lemma}\label{softeig}
\begin{equation}\label{lsofeq}
\mathcal{L}\smax(A)
=\mathrm{Tr}\Big(\frac{e^{\beta A}}{\mathrm{Tr}(e^{\beta A})}\mathcal{L} A\Big)+\beta\sum_{i=1}^{\sdimn}\mathrm{Tr}\Big(\frac{e^{\beta A}}{\mathrm{Tr}(e^{\beta A})}\Big[\nabla_{e_{i}}A-I_{\sdimn}\mathrm{Tr}\Big(\frac{e^{\beta A}}{\mathrm{Tr}(e^{\beta A})}\nabla_{e_{i}}A\Big)\Big]^{2}\Big).
\end{equation}
Consequently, if $\beta>0$,
$$
\mathcal{L}\smax(A)
\geq
\sum_{j=1}^{\sdimn}p_{j}(A)\mathcal{L}a_{jj},
$$
\end{lemma}

\begin{proof}
Fix $i,k\in\{1,\ldots,\sdimn\}$.  Then
\begin{equation}\label{chat3}
\nabla_{e_{i}}\smax(A)
\stackrel{\eqref{smaxdef}}{=}\frac{\nabla_{e_{i}}\mathrm{Tr}(e^{\beta A})}{\beta\mathrm{Tr}(e^{\beta A})}
=\frac{\mathrm{Tr}(\nabla_{e_{i}}e^{\beta A})}{\beta\mathrm{Tr}(e^{\beta A})}
=\frac{\mathrm{Tr}(e^{\beta A}\nabla_{e_{i}}A)}{\mathrm{Tr}(e^{\beta A})}.
\end{equation}
\begin{equation}\label{chat4}
\begin{aligned}
&\nabla_{e_{i}}\nabla_{e_{i}}\smax(A)
\stackrel{\eqref{chat3}}{=}
\frac{\mathrm{Tr}(e^{\beta A})\nabla_{e_{i}}[\mathrm{Tr}(e^{\beta A}\nabla_{e_{i}}A)] -\mathrm{Tr}(e^{\beta A}\nabla_{e_{i}}A)\nabla_{e_{i}}[\mathrm{Tr}(e^{\beta A})]}
{[\mathrm{Tr}(e^{\beta A})]^{2}}\\
&\stackrel{\eqref{chat3}}{=}\frac{\mathrm{Tr}(e^{\beta A})[\mathrm{Tr}(e^{\beta A}\nabla_{e_{i}}\nabla_{e_{i}}A)+\beta\mathrm{Tr}(e^{\beta A}(\nabla_{e_{i}}A)^{2}]
-\mathrm{Tr}(e^{\beta A}\nabla_{e_{i}}A)\beta\mathrm{Tr}(e^{\beta A}\nabla_{e_{i}}A)}
{[\mathrm{Tr}(e^{\beta A})]^{2}}.
\end{aligned}
\end{equation}
Summing \eqref{chat4} over $1\leq i\leq\sdimn$, we obtain
\begin{equation}\label{chat6}
\Delta\smax(A)
=\frac{\mathrm{Tr}(e^{\beta A}\Delta A)}{\mathrm{Tr}(e^{\beta A})}+\beta\sum_{i=1}^{\sdimn}\frac{\mathrm{Tr}(e^{\beta A}(\nabla_{e_{i}}A)^{2})}{\mathrm{Tr}(e^{\beta A})}
-\beta\sum_{i=1}^{\sdimn}\Big(\frac{\mathrm{Tr}(e^{\beta A}\nabla_{e_{i}}A)}{\mathrm{Tr}(e^{\beta A})}\Big)^{2}.
\end{equation}
Note that $C\colonequals e^{\beta A}/\mathrm{Tr}(e^{\beta A})$ is a symmetric positive definite matrix with trace $1$, so the last two terms can be written as a sum of terms of the form $\mathrm{Tr}(C(D-I_{\sdimn}\mathrm{Tr}(CD))^{2})$, i.e.
\begin{equation}\label{chat6.5}
\Delta\smax(A)
=\mathrm{Tr}\Big(\frac{e^{\beta A}}{\mathrm{Tr}(e^{\beta A})}\Delta A\Big)+\beta\sum_{i=1}^{\sdimn}\mathrm{Tr}\Big(\frac{e^{\beta A}}{\mathrm{Tr}(e^{\beta A})}\Big[\nabla_{e_{i}}A-I_{\sdimn}\mathrm{Tr}\Big(\frac{e^{\beta A}}{\mathrm{Tr}(e^{\beta A})}\nabla_{e_{i}}A\Big)\Big]^{2}\Big).
\end{equation}
Recall: if $P,Q$ are symmetric positive semidefinite real matrices, $\mathrm{Tr}(PQ)=\mathrm{Tr}(P^{1/2}QP^{1/2})$, and $P^{1/2}QP^{1/2}$ is itself a symmetric positive semidefinite matrix, so that $\mathrm{Tr}(PQ)\geq0$.  Since $A$ is symmetric, so is $\nabla_{e_{i}}A$, so that \eqref{chat6.5} implies
\begin{equation}\label{chat7}
\Delta\smax(A)
\geq\mathrm{Tr}\Big(\frac{e^{\beta A}}{\mathrm{Tr}(e^{\beta A})}\Delta A\Big).
\end{equation}
Finally, using \eqref{chat3}, if $\beta>0$, then
\begin{flalign*}
\mathcal{L}\smax(A)
&\stackrel{\eqref{three4.3}}{=}
\Delta\smax(A)-\langle x,\nabla\smax(A)\rangle\\
&\stackrel{\eqref{chat7}}{\geq}
\mathrm{Tr}\Big(\frac{e^{\beta A}}{\mathrm{Tr}(e^{\beta A})}[\Delta A-\langle x,\nabla A\rangle]\Big)
\stackrel{\eqref{three4.3}}{=}\mathrm{Tr}\Big(\frac{e^{\beta A}}{\mathrm{Tr}(e^{\beta A})}\mathcal{L} A\Big).
\end{flalign*}
\end{proof}%

The main strategy used to prove Theorem \ref{thm1} is to demonstrate that $\partial\Omega$ has an eigenvalue larger than $2$, and then to use Lemma \ref{orthlem} below to conclude that $\Omega\times\R$ is unstable for Problem \ref{prob1}.  Proving that the hypothesis of Lemma \ref{orthlem} holds seems to require breaking into several different cases, as we do below in the various cases of the proof of Theorem \ref{thm1}.

\begin{lemma}\label{orthlem}
Let $\Omega\subset\R^{\sdimn}$.  Let $f\colon\Sigma\to\R$ satisfy $f(x)=f(-x)$ for all $x\in\Sigma$, $\int_{\Sigma}(f^{2}+\vnorm{\nabla f}^{2}+\abs{f\mathcal{L}f})\gamma_{\sdimn}(x)\,\d x<\infty$, and $Lf=\delta f$, or $\int_{\Sigma} fLf\gamma_{\sdimn}(x)\,\d x\geq\delta\int_{\Sigma}f^{2}\gamma_{\sdimn}(x)\,\d x$.  Let $g\colon\R^{\adimn}\to\R$ defined by
$$g(x_{1},\ldots,x_{\adimn})\colonequals(x_{\adimn}^{2}-1)f(x_{1},\ldots,x_{\sdimn}),\qquad\forall\,(x_{1},\ldots,x_{\sdimn},x_{\adimn})\in\Sigma\times\R.$$
Then $g(-x)=g(x)$ for all $x\in\R^{\adimn}$, $\int_{\Sigma\times\R}g(x)\gamma_{\sdimn}(x)\,\d x=0$, and the corresponding variation of $\Sigma\times\R$ satisfies
$$\frac{\d^{2}}{\d s^{2}}\Big|_{s=0}\int_{(\Sigma\times\R)^{(s)}}\gamma_{\sdimn}(x)dx
\leq -(\delta-2)\int_{\Sigma\times\R}\abs{g(x)}^{2}\gamma_{\sdimn}(x)dx.
$$
\end{lemma}
\begin{proof}
The property $g(x)=g(-x)$ for all $x\in\R^{\adimn}$ is evident.  By Fubini's Theorem, $g$ automatically preserves the Gaussian volume of $\Omega$, i.e.
$$\int_{\Sigma}(x_{\adimn}^{2}-1)f\gamma_{\sdimn}(x)dx=\int_{\R}(x_{\adimn}^{2}-1)\gamma_{1}(x_{\adimn})\int_{\Sigma'}f\gamma_{\sdimn-1}(x_{1},\ldots,x_{\sdimn})dx_{1}\cdots dx_{\sdimn}=0.$$
Since $f$ does not depend on $x_{\adimn}$, $\langle\nabla f,\nabla(x_{\adimn}^{2}-1)\rangle=0$, and by the product rule for $L$ (Remark \ref{rk20}) we have
\begin{equation}\label{six2}
\begin{aligned}
Lg&\stackrel{\eqref{three4.5}}{=}(x_{\adimn}^{2}-1)Lf+f\mathcal{L}(x_{\adimn}^{2}-1)+\langle\nabla f,\nabla(x_{\adimn}^{2}-1)\rangle\\
&=(x_{\adimn}^{2}-1)Lf+f(2-2x_{\adimn}^{2})
\stackrel{\eqref{three4.5}}{=}(x_{\adimn}^{2}-1)Lf-2g.
\end{aligned}
\end{equation}
Finally, from Lemma \ref{varlem2},
\begin{flalign*}
\frac{\d^{2}}{\d s^{2}}\Big|_{s=0}\int_{(\Sigma\times\R)^{(s)}}\gamma_{\sdimn}(x)dx
&=\int_{\Sigma\times\R}-gLg\gamma_{\sdimn}(x)dx
\stackrel{\eqref{six2}}{=}\int_{\Sigma\times\R}-(x_{\adimn}^{2}-1)^{2}[fLf-2f^{2}]\gamma_{\sdimn}(x)dx\\
&\leq-(\delta-2)\int_{\Sigma\times\R}g^{2}\gamma_{\sdimn}(x)dx.
\end{flalign*}
(The integration by parts here is justified by Lemma \ref{lemma39.79}.)
\end{proof}

The following Lemma says that a minimizer of Problem \ref{prob1} satisfies an a priori bound on the Euclidean volume growth of its boundary.

\begin{lemma}\label{polyvollem}
Let $\Omega\subset\R^{\adimn}$ minimize Problem \ref{prob1} with $\sdimn\geq1$.  Then $\Sigma\colonequals\redb\Omega$ has polynomial volume growth, i.e. there exists a constant $b>0$ such that the Euclidean surface area of $\Sigma$ satisfies
$$\mathrm{Vol}(\{x\in\Sigma\colon\vnorm{x}>r\})\leq b r^{\adimn},\qquad\forall\, r>1.$$
\end{lemma}
\begin{proof}
Denote $B(0,r)\colonequals\{x\in\R^{\adimn}\colon \vnorm{x}<r\}$.  Let $\Omega_{\geq r}\colonequals\Omega\cap B(0,r)^{c}$.  Let $C\subset\R^{\adimn}$ be the complement of a ball centered at the origin such that $\gamma_{\adimn}(\Omega_{\geq r})=\gamma_{\adimn}(C)$.  Define $\Omega'\colonequals C\cup(\Omega\setminus\Omega_{\geq r})$.  Then $\gamma_{\adimn}(\Omega)=\gamma_{\adimn}(\Omega')$, and since $\Omega$ minimizes Problem \ref{prob1}, we have
\begin{flalign*}
&\int_{B(0,r)^{c}\cap \Sigma}\gamma_{\sdimn}(x)+\int_{B(0,r)\cap\Sigma}\gamma_{\sdimn}(x)\,\d x
=\int_{\Sigma}\gamma_{\sdimn}(x)\,\d x\\
&\qquad \leq \int_{\Sigma'}\gamma_{\sdimn}(x)\,\d x
=\int_{B(0,r)^{c}\cap \Sigma'}\gamma_{\sdimn}(x)\,\d x+\int_{B(0,r)\cap\Sigma'}\gamma_{\sdimn}(x)\,\d x\\
&\qquad=\int_{B(0,r)^{c}\cap\Sigma}\gamma_{\sdimn}(x)\,\d x+\int_{B(0,r)\cap\Sigma'}\gamma_{\sdimn}(x)\,\d x.
\end{flalign*}
The last line used the definition of $\Omega'$.  Canceling like terms, we then get
$$\int_{B(0,r)^{c}\cap\Sigma}\gamma_{\sdimn}(x)
\leq \int_{B(0,r)^{c}\cap \Sigma'}\gamma_{\sdimn}(x)\,\d x
\leq 2\int_{\partial B(0,r)^{c}}\gamma_{\sdimn}(x)\,\d x.$$
The last inequality used the definition of $\Omega'$.  So, there exists $b=b_{\adimn}$ such that, $\forall$ $r>1$
\begin{equation}\label{cdb}
\int_{B(0,r)^{c}\cap \Sigma}\gamma_{\sdimn}(x)\leq b r^{\adimn}e^{-r^{2}/2}.
\end{equation}
It then follows that $\partial\Omega$ has polynomial volume growth.  Denote $\Sigma\cap B(0,r)^{c}\cap (B(0,N))\equalscolon\Sigma_{r\leq N}$, so that
\begin{equation}\label{cdq}
\begin{aligned}
\gamma_{\sdimn}(\Sigma_{r\leq N})
&=\int_{r}^{N}e^{-t^{2}/2}\frac{\d}{\d t}\mathrm{Vol}(\Sigma_{t\leq N+1})\frac{\d t}{(2\pi)^{\frac{\sdimn}{2}}}
\geq\int_{r}^{N}e^{-t^{2+\epsilon}/2}\frac{\d}{\d t}\mathrm{Vol}(\Sigma_{t\leq N+1})\frac{\d t}{(2\pi)^{\frac{\sdimn}{2}}}\\
&=e^{-r^{2+\epsilon}/2}\mathrm{Vol}(\Sigma_{r\leq N+1})\frac{1}{(2\pi)^{\frac{\sdimn}{2}}}-e^{-N^{2+\epsilon}/2}\mathrm{Vol}(\Sigma_{N\leq N+1})\frac{1}{(2\pi)^{\frac{\sdimn}{2}}}\\
&\qquad\qquad\qquad\qquad+\int_{r}^{N}te^{-t^{2+\epsilon}/2}\mathrm{Vol}(\Sigma_{t\leq N+1})\frac{\d t}{(2\pi)^{\frac{\sdimn}{2}}}.
\end{aligned}
\end{equation}
The last term is nonnegative and finite.  Letting $N\to\infty$, the second term goes to zero by \eqref{cdb}.  Then letting $\epsilon\to0^{+}$, \eqref{cdq} becomes
$$\mathrm{Vol}(\Sigma\cap B(0,r)^{c})\frac{1}{(2\pi)^{\frac{\sdimn}{2}}}\stackrel{\eqref{cdb}}{\leq} e^{r^{2}/2} b r^{\adimn}e^{-r^{2}/2}= br^{\adimn}.$$
\end{proof}

\section{Classification of Stable Self-Shrinkers}

\begin{theorem}\label{thm3}
Let $\Omega\subset\R^{\adimn}$ minimize Problem \ref{prob1}.  Assume that $\Sigma\colonequals\redb\Omega$ satisfies
\begin{equation}\label{one2}
H(x)=\langle x,N(x)\rangle,\qquad\forall\,x\in\Sigma.
\end{equation}
Then there exists an integer $0\leq k\leq\sdimn$ such that, after rotating $\Omega$,
$$\Sigma= \sqrt{k}S^{k}\times\R^{\sdimn-k}.$$
\end{theorem}
\begin{proof}

\textbf{Case 1.}  Assume $\partial\Omega$ is compact.

We repeat the proof of \cite[Proposition 1.5]{heilman17}.  Let $H$ be the mean curvature of $\Sigma$.  If $H\geq0$ on $\Sigma$, then Huisken's classification \cite{huisken90,huisken93} \cite[Theorem 0.17]{colding12a} of compact surfaces satisfying \eqref{one2} implies that $\Sigma$ is a round sphere ($\Sigma=\sqrt{n}S^{n}$).  So, we may assume that $H$ changes sign on $\Sigma$.  From \eqref{three9} with $\scon=0$, $LH=2H$.  Since $H$ changes sign, $2$ is not the largest eigenvalue of $L$, by spectral theory \cite[Lemma 6.5]{zhu16} (e.g. using that $(L-\vnormt{A}^{2}-2)^{-1}$ is a compact operator).  That is, there exists a $C^{2}$ function $g\colon\Sigma\to\R$ and there exists $\pcon>2$ such that $Lg=\pcon g$.  Moreover, $g>0$ on $\Sigma$.  Since $g>0$ and $\Sigma=-\Sigma$, it follows by \eqref{three4.5} that $g(x)+g(-x)$ is an eigenfunction of $L$ with eigenvalue $\pcon$.  That is, we may assume that $g(x)=g(-x)$ for all $x\in\Sigma$.

From the second variation formula, Lemma \ref{varlem2},
$$\frac{\d^{2}}{\d s^{2}}|_{s=0}\int_{\Sigma^{(s)}}\gamma_{\sdimn}(x)\,\d x
=-\int_{\Sigma}f(x)Lf(x)\gamma_{\sdimn}(x)\,\d x.$$

So, to complete the proof, it suffices by Lemma \ref{varlem2} to find a $C^{2}$ function $f$ such that
\begin{itemize}
\item $f(x)=f(-x)$ for all $x\in\Sigma$.  ($f$ preserves symmetry.)
\item $\int_{\Sigma}f(x)\gamma_{\sdimn}(x)\,\d x=0$.  ($f$ preserves Gaussian volume.)
\item $\int_{\Sigma}f(x)Lf(x)\gamma_{\sdimn}(x)\,\d x>0$.  ($f$ decreases Gaussian surface area.)
\end{itemize}
We choose $g$ as above so that $Lg=\pcon g$, $\pcon>2$ and so that $\int_{\Sigma}(H(x)+g(x))\gamma_{\sdimn}(x)\,\d x=0$.  (Since $H$ changes sign and $g\geq0$, $g$ can satisfy the last equality by multiplying it by an appropriate constant.)  We then define $f\colonequals H+g$.  Then $f$ satisfies the first two properties.  So, it remains to show that $f$ satisfies the last property.  Note that, since $H$ and $g$ have different eigenvalues, they are orthogonal, i.e. $\int_{\Sigma}H(x)g(x)\gamma_{\sdimn}(x)\,\d x=0$.  Therefore,
\begin{flalign*}
\int_{\Sigma}f(x)Lf(x)\gamma_{\sdimn}(x)\,\d x
&=\int_{\Sigma}(H(x)+g(x))(2H(x)+\pcon g(x))\gamma_{\sdimn}(x)\,\d x\\
&=2\int_{\Sigma}(H(x))^{2}\gamma_{\sdimn}(x)\,\d x
+\pcon\int_{\Sigma}(g(x))^{2}\gamma_{\sdimn}(x)\,\d x
>0.
\end{flalign*}
(Since $H(x)=\langle x,N(x)\rangle$ for all $x\in\Sigma$, and $\Sigma$ is compact, both $\int_{\Sigma}(H(x))^{2}\gamma_{\sdimn}(x)\,\d x$ and $\int_{\Sigma}H(x)\gamma_{\sdimn}(x)\,\d x$ are finite a priori.)  We conclude that $H$ cannot change sign, i.e. we must have $\Sigma=\sqrt{n}S^{n}$ in this case.

\textbf{Case 2.}  Assume $\partial\Omega$ is not compact, and $H$ is not identically zero.

This case follows the argument of Case 1 \cite[Lemmas 9.44 and 9.45]{colding12a}, \cite[Proposition 6.11]{zhu16}.  If $H$ does not change sign, Huisken's classification \cite[Theorem 0.4]{zhu16} now implies that there exists $0\leq k\leq\sdimn$ such that, after rotating $\Omega$, $\partial\Omega=\sqrt{k}S^{k}\times\R^{\sdimn-k}$.  If $H$ changes sign, then instead of asserting the existence of $g$, one approximates $g$ by a sequence of Dirichlet eigenfunctions on the intersection of $\Sigma$ with large compact balls.  We then repeat the proof of Case 1.  The eigenfunctions $H$ and $g$ are not necessarily orthogonal in the non-compact case, but their inner product $\int_{\Sigma}H(x)g(x)\gamma_{\sdimn}(x)\,\d x$ can be made to satisfy $\abs{\int_{\Sigma}H(x)g(x)\gamma_{\sdimn}(x)\,\d x}\leq\epsilon(\int_{\Sigma}H^{2}\gamma_{\sdimn}(x)\,\d x)^{1/2}(\int_{\Sigma}g^{2}\gamma_{\sdimn}(x)\,\d x)^{1/2}$ for arbitrary $\epsilon>0$ \cite[Equation 6.41]{zhu16}, so the above argument works, with this change.  Also, the integration by parts is justified in Lemmas \ref{lemma39.79} and \ref{lemma39.1}.

\textbf{Case 3.}  Assume $H$ is identically zero.

In this final remaining case, we have from \eqref{one2} that $H(x)=\langle x,N(x)\rangle=0$ for all $x\in\Sigma$.  That is, $\Sigma$ is a (minimal) cone.  We eliminate this case in Section \ref{seccone}.

\end{proof}

\section{Classification of Stable Convex Sets}

\begin{theorem}\label{cor7}
Let $\Omega\subset\R^{\adimn}$.  Assume that $\Omega\times\R$ minimizes Problem \ref{prob1}.   (From Lemma \ref{varlem}, there exists $\scon\in\R$ such that $H(x)=\langle x,N(x)\rangle+\scon$ $\forall$ $x\in\Sigma$.)  Assume that $H\geq0$ on $\Sigma\colonequals\redb\Omega$ and $\scon\geq0$.  Then, after rotating $\Omega$, 
$$\partial\Omega= \lambda S^{0}\times\R^{n}.$$
\end{theorem}
\begin{proof}
By \eqref{three9},
\begin{equation}\label{six0}
\int_{\partial\Omega}H LH\gamma_{\sdimn}(x)\,\d x=\int_{\partial\Omega}[2H^{2}+\scon H\vnorm{A}^{2}]\gamma_{\sdimn}(x)\,\d x.
\end{equation}
(The quantity $\int_{\partial\Omega}H LH\gamma_{\sdimn}(x)\,\d x$ is finite a priori by Lemma \ref{lemma39.1}; note also $\partial\Omega$ has polynomial volume growth by Lemma \ref{polyvollem}.)   So,
\begin{equation}\label{six1}
\int_{\partial\Omega}H LH\gamma_{\sdimn}(x)\,\d x\geq\int_{\partial\Omega}2H^{2}\gamma_{\sdimn}(x)\,\d x,
\end{equation}
with equality only when $H\vnorm{A}=0$ identically, i.e. when $H=0$ identically (since $\vnorm{A}=0$ implies $H=0$).  If $H=0$ on all of $\Sigma$, then by \eqref{three9},
$$0=LH=2H+\lambda\vnorm{A}^{2}=\lambda\vnorm{A}^{2}.$$
Since $\lambda>0$ we conclude that $\vnorm{A}=0$ on all of $\Sigma$.  In summary, the inequality in \eqref{six1} is strict, unless $\vnorm{A}=0$ on $\Sigma$.

In the former case of strict inequality in \eqref{six1}, let $f\colonequals H$ in Lemma \ref{orthlem}.  Note that $H(x)=H(-x)$ for all $x\in\Sigma$.  Then Lemma \ref{orthlem} and \eqref{six1} imply that the variation of $\Omega\times\R$ correspond to $g$ satisfies
$$\frac{\d^{2}}{\d s^{2}}\Big|_{s=0}\int_{(\Sigma\times\R)^{(s)}}\gamma_{\adimn}(x)dx<0.$$
This inequality violates the minimality of $\Omega$, achieving a contradiction, demonstrating that we must have $\vnorm{A}=0$ on $\Sigma$.  In the latter case, we must then have $\Sigma= rS^{0}\times\R^{\sdimn}$ for some $r>0$.  Since $0=H=\langle x,N\rangle+\lambda$, we conclude that $r=\lambda$.
\end{proof}

Note that Theorem \ref{cor7} proved a stronger result than we need to prove Theorem \ref{thm1}. Theorem \ref{cor7} shows that the only $\Omega$ with $\Omega\times\R$ stable for Gaussian surface area with $H\geq0$ and $\lambda>0$ (or $H\leq0$ and $\lambda<0$) is a symmetric slab or its complement.

\section{Non-Mean Convex Sets are Unstable}

In order to prove instability of noncompact sets, we will need to integrate by parts in the second variation formula in Lemma \ref{varlem2}.  The noncompactness, together with the low-dimensional singularities on the surface (see Lemma \ref{lemma51}) mean that integrating by parts is nontrivial, hence the results below.

\begin{lemma}[\embolden{Integration by Parts}, {\cite[Corollary 3.10]{colding12a}, \cite[Lemma 5.4]{zhu16}, \cite[Lemma 4.4]{heilman17}}]\label{lemma39.7}
Let $\Sigma\subset\R^{\adimn}$ be an $\sdimn$-dimensional hypersurface.  Let $f,g\colon\Sigma\to\R$.  Assume that $f$ is a $C^{2}$ function and $g$ is a $C^{2}$ function with compact support.  Then
$$\int_{\Sigma}f\mathcal{L}g\gamma_{\sdimn}(x)dx=\int_{\Sigma}g\mathcal{L}f\gamma_{\sdimn}(x)dx=-\int_{\Sigma}\langle\nabla f,\nabla g\rangle\gamma_{\sdimn}(x)dx.$$
\end{lemma}

\begin{cor}[\embolden{Integration by Parts, Version 2} {\cite[Lemma 5.4]{zhu16}, \cite[Corollary 5.10]{heilman17}}]\label{lemma39.79}
Let $\Omega\subset\R^{\adimn}$.  Let $f,g\colon\redA\to\R$ be $C^{2}$ functions.  Suppose the Hausdorff dimension of $\partial\Omega\setminus\redA$ is at most $\sdimn-7$.  Assume that
$$\int_{\Sigma}\vnormt{\nabla f}\vnormt{\nabla g}\gamma_{\sdimn}(x)dx<\infty,\qquad
\int_{\Sigma}\abs{f\mathcal{L}g}\gamma_{\sdimn}(x)dx<\infty,\qquad
\int_{\Sigma}\vnormt{f\nabla g}^{6/5}\gamma_{\sdimn}(x)dx<\infty.
$$
Then
$$\int_{\Sigma}f\mathcal{L}g\gamma_{\sdimn}(x)dx
=-\int_{\Sigma}\langle\nabla f,\nabla g\rangle\gamma_{\sdimn}(x)dx.
$$
\end{cor}

\begin{cor}[{\cite[Lemma 6.2]{zhu16}, \cite[Corollary 5.5]{heilman17}}]\label{cor5}
Let $\Omega\subset\R^{\adimn}$ and let $\Sigma\colonequals\partial\Omega$.  Assume $\pcon(\Sigma)<\infty$.  Suppose $g\colon\Sigma\to\R$ is a $C^{2}$ function with $g>0$ and $Lg=\pcon g$.  Assume that the Hausdorff dimension of $\partial\Omega\setminus\redA$ is at most $\sdimn-7$.  Then for any $k\geq0$, $\vnormt{A}\vnormt{x}^{k}\in L_{2}(\Sigma,\gamma_{\sdimn})$
\end{cor}

\begin{lemma}[{\cite[Theorem 9.36]{colding12a}, \cite[Lemma 5.6]{heilman17}}]\label{lemma39.1}
Let $\Sigma\subset\R^{\adimn}$ be a connected, orientable $C^{\infty}$ hypersurface with polynomial volume growth and with possibly nonempty boundary.  Assume $\exists$ $\scon\in\R$ such that $H(x)=\langle x,N(x)\rangle+\scon$ for all $x\in\Sigma$.  Let $\pcon\colonequals\pcon(\Sigma)$.   Assume $\pcon(\Sigma)<\infty$.  Then
$$\vnormt{\nabla H}\in L_{2}(\Sigma,\gamma_{\sdimn}),\qquad\vnormt{A}\abs{H}\in L_{2}(\Sigma,\gamma_{\sdimn}).$$
$$\int_{\Sigma}\abs{H\mathcal{L}H}\gamma_{\sdimn}(x)dx<\infty.$$
\end{lemma}

\begin{theorem}\label{thm5}
Let $\Omega\subset\R^{\adimn}$.  Suppose $H$ changes sign on $\Sigma\colonequals\redb\Omega$ (that is, $\Sigma$ is not mean convex).  Then $\Omega\times\R$ does not minimize Problem \ref{prob1}.
\end{theorem}
\begin{proof}
Assume for the sake of contradiction that $\Omega\times\R$ minimizes Problem \ref{prob1}.  From Lemma \ref{varlem}, there exists $\scon\in\R$ such that
\begin{equation}\label{hdeq}
H(x)=\langle x,N(x)\rangle+\scon,\qquad\forall\,x\in\Sigma.
\end{equation}
The case $\scon=0$ was already treated in Theorem \ref{thm3}.  We may therefore assume that $\scon\neq0$.  Replacing $\Omega$ with $\Omega^{c}$ if necessary (i.e. changing the direction of $N$, which changes the sign of $H$ and $N$ and therefore of $\scon$), we may assume that $\scon>0$.

Since $H$ is not mean convex, $H$ changes sign, i.e. there exists a set of positive measure on $\Sigma$ where $H>0$.  Let $h\colonequals\max(H,0)$.  By \eqref{three9},
\begin{equation}\label{six0b}
\int_{\Sigma}h Lh\gamma_{\sdimn}(x)\,\d x=\int_{\Sigma}[2h^{2}+\scon h\vnorm{A}^{2}]\gamma_{\sdimn}(x)\,\d x.
\end{equation}
(Note that $\Sigma$ has polynomial volume growth by Lemma \ref{polyvollem}.  Then, by Lemma \ref{lemma39.1} and Corollary \ref{lemma39.79}, $H$ and $h$ are well-defined, the integrals in \eqref{six0b} are finite, and we can integrate by parts with $H$ or $h$.)  Since $\scon>0$,
\begin{equation}\label{six1b}
\int_{\partial\Omega}h Lh\gamma_{\sdimn}(x)\,\d x>\int_{\partial\Omega}2h^{2}\gamma_{\sdimn}(x)\,\d x.
\end{equation}
Note that equality cannot occur in \eqref{six1b} since equality would only occur when $\vnorm{A}=0$ on the set where $H>)$, but this cannot happen since $\vnorm{A}=0$ implies $H=0$).

Note that $h$ satisfies $h(x)=h(-x)$ for all $x\in\Sigma$ and $h\neq0$.    Then Lemma \ref{orthlem} and \eqref{six1b} imply that the variation of $\Omega\times\R$ correspond to $h$ satisfies
$$\frac{\d^{2}}{\d s^{2}}\Big|_{s=0}\int_{(\Sigma\times\R)^{(s)}}\gamma_{\sdimn}(x)dx<0.$$
This inequality violates the minimality of $\Omega\times\R$, achieving a contradiction, and completing the proof.
\end{proof}

\section{Curvature Bounds}

We will eventually reduce the mean-convex case of Theorem \ref{thm1} to the convex case.  Doing so requires using a (smoothed version of) the largest eigenvalue of the second fundamental form $A$ (see \eqref{smaxdef}) in the second variation formula from Lemma \ref{varlem2}.  In order to use \eqref{smaxdef} in Lemma \ref{varlem2}, we need to integrate by parts, which then requires an a priori curvature bound to hold.  This curvature bound (see Lemma \ref{lemma39.2}) was known to hold for mean-convex self-shrinkers, i.e. when $H=\langle x,N\rangle>0$ \cite{colding12a}, and a few small adjustments to their argument allow the bound to apply in the case $H=\langle x,N\rangle+\lambda$ with $\lambda<0$ and $H\geq0$.

For any hypersurface $\Sigma\subset\R^{\adimn}$, we define
\begin{equation}\label{seven0}
\pcon=\pcon(\Sigma)
\colonequals\sup_{f\in C_{0}^{\infty}(\Sigma)}\frac{\int_{\Sigma}fLf\gamma_{\sdimn}(x)dx}{\int_{\Sigma}f^{2}\gamma_{\sdimn}(x)dx}.
\end{equation}

\begin{lemma}[{\cite[Lemma 5.1]{heilman17}}]\label{lemma95.1}
If $\Omega\subset\R^{\adimn}$ minimizes Problem \ref{prob1}, then $\pcon(\redA)<\infty$.
\end{lemma}

\begin{lemma}[\embolden{Simons-type inequality}, {\cite{simons68,colding12a,cheng15,zhu16}}]\label{lemma38}
Let $\Sigma$ be a $C^{\infty}$ orientable hypersurface.  Let $\scon\in\R$.  Suppose $H(x)=\langle x,N(x)\rangle+\scon$, $\forall\,x\in\Sigma$.  Then
\begin{equation}\label{three30}
\begin{aligned}
\vnormt{A}L\vnormt{A}
&=2\vnormt{A}^{2}-\scon\langle A^{2},A\rangle+\vnormt{\nabla A}^{2}-\vnormt{\nabla \vnormt{A}}^{2}\\
&\geq 2\vnormt{A}^{2}-\scon\langle A^{2},A\rangle.
\end{aligned}
\end{equation}
\end{lemma}

\begin{lemma}[{\cite[Lemma 10.2]{colding12a}}]\label{lemma39.9}
Let $\Sigma\subset\R^{\adimn}$ be any $\sdimn$-dimensional hypersurface.  Then
\begin{equation}\label{two12}
\Big(1+\frac{2}{\adimn}\Big)\vnormt{\nabla\vnormt{A}}^{2}
\leq \vnormt{\nabla A}^{2}+\frac{2\sdimn}{\adimn}\vnormt{\nabla H}^{2}.
\end{equation}
\end{lemma}

\begin{lemma}[{\cite[Proposition 10.14]{colding12a}, \cite[Lemma 5.9]{heilman17}}]\label{lemma39.2}
Let $\Omega$ minimize problem \ref{prob1}.  Assume $\pcon(\redb\Omega)<\infty$.  Assume that $H\geq0$ on $\Sigma$  (From Lemma \ref{varlem}, $H(x)=\langle x,N(x)\rangle+\scon$ for all $x\in\Sigma$).  Assume $\scon<0$.  Then
$$\int_{\Sigma}(\vnormt{A}^{2}+\vnormt{A}^{4}+\vnormt{\nabla\vnormt{A}}^{2}+\vnormt{\nabla A}^{2})\gamma_{\sdimn}(x)dx<\infty.$$
\end{lemma}
\begin{proof}
Since $H\geq0$, and 
$$(\Delta -\langle x,\nabla \rangle-1)H\stackrel{\eqref{three4.5}}{=}(L-2-\vnorm{A}^{2})H\stackrel{\eqref{three9}}{=}(\lambda-H)\vnorm{A}^{2}\leq0,$$
the maximum principle \cite[Theorem 6.4.2]{evans98} (or in this case, the minimum principle applied to $L-2-\vnorm{A}^{2}$) implies that $H>0$ on $\Sigma$.  Then $\log H$ is well-defined, so 
\begin{equation}\label{logeq}
\begin{aligned}
&\mathcal{L}\log H
\stackrel{\eqref{three4.3}}{=}\sum_{i=1}^{\sdimn}\nabla_{e_{i}}\left(\frac{\nabla_{e_{i}}H}{H}\right)-\frac{\langle x,\nabla H\rangle}{H}
=\sum_{i=1}^{\sdimn}\frac{-(\nabla_{e_{i}}H)^{2}}{H^{2}}+\frac{\mathcal{L}H}{H}
=-\vnormt{\nabla\log H}^{2}+\frac{\mathcal{L}H}{H}\\
&\stackrel{\eqref{three4.5}}{=}-\vnormt{\nabla\log H}^{2}+\frac{LH -\vnormt{A}^{2}H-H}{H}
\stackrel{\eqref{three9}}{=}-\vnormt{\nabla\log H}^{2}+\frac{2H+\scon\vnormt{A}^{2} -\vnormt{A}^{2}H-H}{H}\\
&\,\,=-\vnormt{\nabla\log H}^{2}+(1-\vnormt{A}^{2})+\scon\frac{\vnormt{A}^{2}}{H}.
\end{aligned}
\end{equation}

Note that $\int_{\Sigma}\vnormt{A}^{2}\gamma_{\sdimn}(x)dx<\infty$ by Corollary \ref{cor5}.  Let $\phi\in C_{0}^{\infty}(\Sigma)$.  Integrating by parts with Lemma \ref{lemma39.7},
\begin{flalign*}
&\int_{\Sigma}\langle\nabla\phi^{2},\nabla\log H\rangle\gamma_{\sdimn}(x)dx
=-\int_{\Sigma}\phi^{2}\mathcal{L}\log H\gamma_{\sdimn}(x)dx\\
&\qquad\qquad\qquad\stackrel{\eqref{logeq}}{=}\int_{\Sigma}\phi^{2}\Big(\vnormt{\nabla\log H}^{2}+(-1+\vnormt{A}^{2})-\scon\frac{\vnormt{A}^{2}}{H}\Big)\gamma_{\sdimn}(x)dx.
\end{flalign*}
From the AMGM inequality, $\abs{\langle\nabla\phi^{2},\nabla\log H\rangle}\leq\vnormt{\nabla\phi}^{2}+\phi^{2}\vnormt{\nabla\log H}^{2}$, so that
\begin{equation}\label{two57z}
\int_{\Sigma}\phi^{2}\Big(\vnormt{A}^{2}-\scon\frac{\vnorm{A}^{2}}{H}\Big)\gamma_{\sdimn}(x)dx
\leq\int_{\Sigma}\Big[\vnormt{\nabla \phi}^{2}+\phi^{2}\Big]\gamma_{\sdimn}(x)dx.
\end{equation}
Recall that we assumed $\scon<0$ so the integral on the left is nonnegative.  Therefore,
\begin{equation}\label{two57}
\int_{\Sigma}\phi^{2}\vnormt{A}^{2}\gamma_{\sdimn}(x)dx
\leq\int_{\Sigma}\Big[\vnormt{\nabla \phi}^{2}+\phi^{2}\Big]\gamma_{\sdimn}(x)dx.
\end{equation}
Let $0<\epsilon<1/2$ to be chosen later.  Using now $\phi\colonequals\eta\vnormt{A}$ in \eqref{two57}, where $\eta\in C_{0}^{\infty}(\Sigma)$, $\eta\geq0$, and using the AMGM inequality in the form $2ab\leq \epsilon a^{2}+b^{2}/\epsilon$, $a,b>0$,
\begin{equation}\label{two13}
\begin{aligned}
&\int_{\Sigma}\eta^{2}\vnormt{A}^{4}\gamma_{\sdimn}(x)dx\\
&\leq\int_{\Sigma}\Big[\eta^{2}\vnormt{\nabla\vnormt{A}}^{2}+2\eta\vnormt{A}\vnormt{\nabla\vnormt{A}}\vnormt{\nabla\eta}
+\vnormt{A}^{2}\vnormt{\nabla\eta}^{2}+\eta^{2}\vnormt{A}^{2}\Big]\gamma_{\sdimn}(x)dx\\
&\leq\int_{\Sigma}\Big[(1+\epsilon)\eta^{2}\vnormt{\nabla\vnormt{A}}^{2}
+\vnormt{A}^{2}\vnormt{\nabla\eta}^{2}(1+1/\epsilon)+\eta^{2}\vnormt{A}^{2}\Big]\gamma_{\sdimn}(x)dx.
\end{aligned}
\end{equation}

Using the product rule for $\mathcal{L}$, and that $\mathcal{L}=L-\vnormt{A}^{2}-1$
\begin{flalign*}
\frac{1}{2}\mathcal{L}\vnormt{A}^{2}
&=\vnormt{\nabla\vnormt{A}}^{2}+\vnormt{A}\mathcal{L}\vnormt{A}
=\vnormt{\nabla\vnormt{A}}^{2}+\vnormt{A}\Big(L\vnormt{A}-\vnormt{A}^{3}-\vnormt{A}\Big)\\
&\stackrel{\eqref{three30}}{=}\vnormt{\nabla\vnormt{A}}^{2}+2\vnormt{A}^{2}-\scon\langle A^{2},A\rangle
+\vnormt{\nabla A}^{2}-\vnormt{\nabla \vnormt{A}}^{2}-\vnormt{A}^{4}-\vnormt{A}^{2}\\
&=\vnormt{\nabla A}^{2}+\vnormt{A}^{2}-\vnormt{A}^{4}-\scon\langle A^{2},A\rangle\\
&\stackrel{\eqref{two12}}{\geq}
\Big(1+\frac{2}{\adimn}\Big)\vnormt{\nabla\vnormt{A}}^{2}
-\frac{2\sdimn}{\adimn}\vnormt{\nabla H}^{2}
+\vnormt{A}^{2}-\vnormt{A}^{4}-\scon\langle A^{2},A\rangle.
\end{flalign*}
Multiplying this inequality by $\eta^{2}$ and integrating by parts with Lemma \ref{lemma39.7},
\begin{flalign*}
&-2\int_{\Sigma}\eta\vnormt{A}\langle\nabla\eta,\nabla\vnormt{A}\rangle\gamma_{\sdimn}(x)dx
=-\frac{1}{2}\int_{\Sigma}\langle\nabla\eta^{2},\nabla\vnormt{A}^{2}\rangle\gamma_{\sdimn}(x)dx
=\frac{1}{2}\int_{\Sigma}\eta^{2}\mathcal{L}\vnormt{A}^{2}\gamma_{\sdimn}(x)dx\\
&\qquad\geq\int_{\Sigma}\eta^{2}\Big(\Big(1+\frac{2}{\adimn}\Big)\vnormt{\nabla\vnormt{A}}^{2}
-\frac{2\sdimn}{\adimn}\vnormt{\nabla H}^{2}
-\vnormt{A}^{4}-\scon\langle A^{2},A\rangle\Big)\gamma_{\sdimn}(x)dx.
\end{flalign*}
(We removed the $\vnormt{A}^{2}$ term since doing so only decreases the quantity on the right.)  Rearranging this inequality and then using the AMGM inequality in the form $b^{2}/\epsilon-2ab\geq -\epsilon a^{2}$,
\begin{equation}\label{two14}
\begin{aligned}
&\int_{\Sigma}\Big(\eta^{2}\vnormt{A}^{4}+\scon\eta^{2}\langle A^{2},A\rangle
+\frac{2\sdimn}{\adimn}\eta^{2}\vnormt{\nabla H}^{2}+\frac{1}{\epsilon}\eta^{2}\vnormt{A}^{2}\vnormt{\nabla \eta}^{2}\Big)\gamma_{\sdimn}(x)dx\\
&\qquad\qquad\geq \Big(1+\frac{2}{\adimn}-\epsilon\Big)\int_{\Sigma}\eta^{2}\vnormt{\nabla\vnormt{A}}^{2}\gamma_{\sdimn}(x)dx.
\end{aligned}
\end{equation}

Substituting \eqref{two14} into \eqref{two13},
\begin{flalign*}
&\int_{\Sigma}\eta^{2}\vnormt{A}^{4}\gamma_{\sdimn}(x)dx
\leq \frac{1+\epsilon}{1+\frac{2}{\adimn}-\epsilon}\int_{\Sigma}\eta^{2}\vnormt{A}^{4}\gamma_{\sdimn}(x)dx\\
&\quad
+10\int_{\Sigma}\Big[\eta^{2}\vnormt{\nabla H}^{2}+\eta^{2}(1+1/\epsilon)(1+\vnormt{\nabla\eta}^{2})\vnormt{A}^{2}+\abs{\scon}\eta^{2}\Big(\abs{\langle A,A^{2}\rangle}\Big)\Big]\gamma_{\sdimn}(x)dx.
\end{flalign*}
Note that $\abs{\langle A,A^{2}\rangle}\leq\vnormt{A}^{3}$.  Using the AMGM inequality in the form $2a^{3}\leq a^{4}\epsilon+a^{2}/\epsilon$, we then get
\begin{flalign*}
\int_{\Sigma}\eta^{2}\vnormt{A}^{4}\gamma_{\sdimn}(x)dx
&\leq \Big(10\epsilon\abs{\scon}+\frac{1+\epsilon}{1+\frac{2}{\adimn}-\epsilon}\Big)\int_{\Sigma}\eta^{2}\vnormt{A}^{4}\gamma_{\sdimn}(x)dx\\
&\qquad
+10\int_{\Sigma}\eta^{2}\Big(\vnormt{\nabla H}^{2}+(1+(1+\abs{\scon})/\epsilon)(1+\vnormt{\nabla\eta}^{2})\vnormt{A}^{2}\Big)\gamma_{\sdimn}(x)dx.
\end{flalign*}

Now, choose $\epsilon<1/(20(n+1)(\abs{\scon}+1))$, so that $10\abs{\scon}\epsilon+\frac{1+\epsilon}{1+\frac{2}{\adimn}-\epsilon}<1$.  We then can move the $\eta^{2}\vnormt{A}^{4}$ term on the right side to the left side to get some $c_{\epsilon}>0$ such that
\begin{flalign*}
\int_{\Sigma}\eta^{2}\vnormt{A}^{4}\gamma_{\sdimn}(x)dx
&\leq c_{\epsilon}\int_{\Sigma}\eta^{2}\Big[\vnormt{\nabla H}^{2}+\vnormt{A}^{2}(1+\vnormt{\nabla\eta}^{2})\Big]\gamma_{\sdimn}(x)dx\\
&\leq c_{\epsilon}\int_{\Sigma}\eta^{2}\Big[\vnormt{A}^{2}(1+\vnormt{x}^{2}+\vnormt{\nabla\eta}^{2})\Big]\gamma_{\sdimn}(x)dx.
\end{flalign*}
In the last line, we used the inequality $\vnormt{\nabla H}^{2}\leq\vnormt{A}^{2}\vnormt{x}^{2}$.  This follows by Lemma \ref{varlem}, since $H(x)=\langle x,N\rangle+\scon$, so for any $1\leq i\leq\sdimn$, $\nabla_{e_{i}}H(x)=-\sum_{j=1}^{\sdimn}a_{ij}\langle x,e_{j}\rangle$.

We now choose a sequence of $\eta=\eta_{r}$ increasing to $1$ as $r\to\infty$ so that the $\vnormt{\nabla\eta}^{2}$ term vanishes.  This is possible due to the assumptions that $\delta(\redb\Omega)<\infty$ and the Hausdorff dimension of $\partial\Omega\setminus\redA$ is at most $n-4$.  Such functions are constructed and this estimate is made in \cite[Lemma 6.4]{zhu16}:
\begin{equation}\label{two16}
\int_{\Sigma}\vnormt{A}^{2}\vnormt{\nabla\eta_{r}}^{2}\eta_{r}^{2}\gamma_{\sdimn}(x)dx
\leq c(\delta)(r^{\sdimn}e^{-(r-4)^{2}/4}+r^{-1}),\qquad\forall r>1.
\end{equation}

It therefore follows from Corollary \ref{cor5} applied to $\Sigma=\redA$ that
\begin{equation}\label{two15}
\int_{\Sigma}\vnormt{A}^{4}\gamma_{\sdimn}(x)dx<\infty.
\end{equation}
It then follows from \eqref{two14} that $\int_{\Sigma}\vnormt{\nabla\vnormt{A}}^{2}\gamma_{\sdimn}(x)dx<\infty.$

Finally, multiplying the above equality
$\mathcal{L}\vnormt{A}^{2}=2\vnormt{\nabla A}^{2}+2\vnormt{A}^{2}-2\vnormt{A}^{4}-2\scon\langle A^{2},A\rangle$ by $\eta^{2}$ and integrating by parts with Lemma \ref{lemma39.7}, we get
\begin{flalign*}
&2\int_{\Sigma}\eta^{2}(\vnormt{\nabla A}^{2}-\vnormt{A}^{4}-\scon\langle A^{2},A\rangle)\gamma_{\sdimn}(x)dx
\leq \int_{\Sigma}\eta^{2}\mathcal{L}\vnormt{A}^{2}\gamma_{\sdimn}(x)dx\\
&\qquad=-4\int_{\Sigma}\eta\vnormt{A}\langle\nabla\eta,\nabla\vnormt{A}\rangle\gamma_{\sdimn}(x)dx
\leq2\int_{\Sigma}\eta^{2}\vnormt{\nabla\vnormt{A}}^{2}+\vnormt{A}^{2}\vnormt{\nabla\eta}^{2}\rangle\gamma_{\sdimn}(x)dx.
\end{flalign*}
Then the $\vnormt{A}^{4}$ integral is finite by \eqref{two15}, the $\vnormt{\nabla\vnormt{A}}^{2}$ integral is finite, the last term has a finite integral by \eqref{two16}, so the integral of $\vnormt{\nabla A}^{2}$ is also finite.
\end{proof}

Recall the definition of $\smax(A)$ from \eqref{smaxdef}.  In order to integrate by parts with $\smax(A)$ via Lemma \ref{lemma39.79}, we need the following Lemma

\begin{lemma}\label{smaxlemma}
Let $\beta>0$.  Then
$$\sup_{\beta>0}\int_{\Sigma}\vnorm{\nabla\smax(A)}^{2}\gamma_{\sdimn}(x)\,\d x<\infty,
\qquad\qquad\sup_{\beta>1}\int_{\Sigma}\abs{\smax(A)\mathcal{L}\smax(A)}\gamma_{\sdimn}(x)\, \d x<\infty,$$
$$\sup_{\beta>1}\int_{\Sigma}\vnorm{\smax(A)\nabla\smax(A)}^{6/5}\gamma_{\sdimn}(x)\, \d x<\infty.$$
\end{lemma}
\begin{proof}
From \eqref{chat3} and Lemma \ref{lemma39.2}, we have for any $\beta>0$
$$
\int_{\Sigma}\vnorm{\nabla\smax(A)}^{2}\gamma_{\sdimn}(x)\,\d x
\leq \int_{\Sigma}\vnorm{\nabla A}^{2}\gamma_{\sdimn}(x)\,\d x
<\infty.
$$
Note that $\max_{\beta}(A)\leq\frac{\log n}{\beta}+\vnorm{A}$.  From \eqref{chat3}, H\"{o}lder's inequality with exponents $3$ and $3/2$, and Lemma \ref{lemma39.2},
\begin{flalign*}
&\int_{\Sigma}\vnorm{\smax(A)\nabla\smax(A)}^{6/5}\gamma_{\sdimn}(x)\,\d x\\
&\qquad\qquad\qquad\leq\Big(\int_{\Sigma}\Big(\frac{\log n}{\beta}+\vnorm{A}\Big)^{18/5}\gamma_{\sdimn}(x)\,\d x\Big)^{1/3}\Big(\int_{\Sigma}\vnorm{\nabla A}^{18/10}\gamma_{\sdimn}(x)\,\d x\Big)^{2/3}
<\infty.
\end{flalign*}

Now, let $\phi\in C_{0}^{\infty}(\Sigma)$, Let $h\colonequals\max(\smax(A),0)$, and integrate by parts with Lemma \ref{lemma39.7},
\begin{equation}\label{ibpeq}
\begin{aligned}
\int_{\Sigma}\phi h\mathcal{L} h\gamma_{\sdimn}(x)dx
&=-\int_{\{x\in\Sigma\colon \smax(A)\geq0\}}\phi\vnorm{\nabla\smax(A)}^{2}\gamma_{\sdimn}(x)dx\\
&\qquad-\int_{\{x\in\Sigma\colon \smax(A)\geq0\}}\smax(A)\langle\nabla\smax(A),\nabla\phi\rangle\gamma_{\sdimn}(x)dx.
\end{aligned}
\end{equation}
Let $\phi$ be an approximation to the identity supported on a ball of radius $r>0$ as in Lemma \ref{lemma39.2}.  Letting $r\to\infty$ and using $\int_{\Sigma}\vnorm{\smax(A)\nabla\smax(A)}^{6/5}\gamma_{\sdimn}(x)\,\d x<\infty$, the last integral in \eqref{ibpeq} goes to zero by e.g. \cite[Corollary 5.3]{zhu16}.  Also, $\nabla\smax(A)\stackrel{\eqref{chat3}}{=}\mathrm{Tr}(e^{\beta A}\nabla A)/\mathrm{Tr}(e^{\beta A})$, so that $\vnorm{\nabla\smax(A)}^{2}\leq\vnorm{\nabla A}^{2}$, so the term $\int_{\Sigma}\phi\vnorm{\nabla\smax(A)}^{2}\gamma_{\sdimn}(x)dx$ in \eqref{ibpeq} converges as $r\to\infty$, since $\int_{\Sigma}\vnorm{\nabla A}^{2}\gamma_{\sdimn}(x)\,\d x<\infty$ by Lemma \ref{lemma39.2}.  We therefore conclude that the first integral in \eqref{ibpeq} converges as $r\to\infty$.  Observe that
\begin{equation}\label{dbpeq}
\begin{aligned}
&\int_{\Sigma}\phi h\mathcal{L}(h)\gamma_{\sdimn}(x)dx
=\int_{\{x\in\Sigma\colon\smax(A)\geq0\}}\phi\,\smax(A)\mathcal{L}(\smax(A))\gamma_{\sdimn}(x)dx\\
&\stackrel{\eqref{lsofeq}}{=}\int_{\{x\in\Sigma\colon\smax(A)\geq0\}}\phi\,\smax(A)\Big[\mathrm{Tr}\Big(\frac{e^{\beta A}}{\mathrm{Tr}(e^{\beta A})}\mathcal{L} A\Big)\\
&\qquad\qquad\qquad\qquad\qquad\qquad+\beta\sum_{i=1}^{\sdimn}\mathrm{Tr}\Big(\frac{e^{\beta A}}{\mathrm{Tr}(e^{\beta A})}\Big[\nabla_{e_{i}}A-I_{\sdimn}\mathrm{Tr}\Big(\frac{e^{\beta A}}{\mathrm{Tr}(e^{\beta A})}\nabla_{e_{i}}A\Big)\Big]^{2}\Big)
\Big]\gamma_{\sdimn}(x)dx\\
&\stackrel{\eqref{three9p}}{=}\int_{\{x\in\Sigma\colon\smax(A)\geq0\}}\phi\,\smax(A)\Big[\mathrm{Tr}\Big(\frac{e^{\beta A}}{\mathrm{Tr}(e^{\beta A})}[2A-\lambda A^{2}-A(\vnorm{A}^{2}+1)]\Big)\\
&\qquad\qquad\qquad\qquad\qquad\qquad+\beta\sum_{i=1}^{\sdimn}\mathrm{Tr}\Big(\frac{e^{\beta A}}{\mathrm{Tr}(e^{\beta A})}\Big[\nabla_{e_{i}}A-I_{\sdimn}\mathrm{Tr}\Big(\frac{e^{\beta A}}{\mathrm{Tr}(e^{\beta A})}\nabla_{e_{i}}A\Big)\Big]^{2}\Big)
\Big]\gamma_{\sdimn}(x)dx.
\end{aligned}
\end{equation}
For any $r,\beta>1$ we can bound all terms except the last one in absolute value by a constant plus
\begin{flalign*}
&\int_{\Sigma}\Big(2\vnorm{A}^{2}+\abs{\scon}\vnorm{A}^{3}+\vnorm{A}^{4}+\vnorm{A}^{3}\Big)\gamma_{\sdimn}(x)dx.
\end{flalign*}
This quantity is finite by Lemma \ref{lemma39.2}.  As observed in \eqref{ibpeq}, the first term in \eqref{dbpeq} converges to a finite value as $r\to\infty$.  Letting then $r\to\infty$, we conclude that
$$\int_{\{x\in\Sigma\colon\smax(A)\geq0\}}\smax(A)\beta\sum_{i=1}^{\sdimn}\mathrm{Tr}\Big(\frac{e^{\beta A}}{\mathrm{Tr}(e^{\beta A})}\Big[\nabla_{e_{i}}A-I_{\sdimn}\mathrm{Tr}\Big(\frac{e^{\beta A}}{\mathrm{Tr}(e^{\beta A})}\nabla_{e_{i}}A\Big)\Big]^{2}\Big)\gamma_{\sdimn}(x)\,\d x<\infty.$$
Repeating the above argument with $h\colonequals\max(-\smax(A),0)$, we also obtain
$$\int_{\{x\in\Sigma\colon\smax(A)\leq0\}}-\smax(A)\beta\sum_{i=1}^{\sdimn}\mathrm{Tr}\Big(\frac{e^{\beta A}}{\mathrm{Tr}(e^{\beta A})}\Big[\nabla_{e_{i}}A-I_{\sdimn}\mathrm{Tr}\Big(\frac{e^{\beta A}}{\mathrm{Tr}(e^{\beta A})}\nabla_{e_{i}}A\Big)\Big]^{2}\Big)\gamma_{\sdimn}(x)\,\d x<\infty.$$
Therefore,
$$\int_{\Sigma}\sum_{i=1}^{\sdimn}\abs{\smax(A)}\beta\sum_{i=1}^{\sdimn}\mathrm{Tr}\Big(\frac{e^{\beta A}}{\mathrm{Tr}(e^{\beta A})}\Big[\nabla_{e_{i}}A-I_{\sdimn}\mathrm{Tr}\Big(\frac{e^{\beta A}}{\mathrm{Tr}(e^{\beta A})}\nabla_{e_{i}}A\Big)\Big]^{2}\Big)\gamma_{\sdimn}(x)\,\d x<\infty.$$
We finally can conclude that
$$\int_{\Sigma}\abs{\smax(A)\mathcal{L}\smax(A)}\gamma_{\sdimn}(x)\, \d x<\infty.$$
All above upper bounds on these integrals did not depend on $\beta>1$, so we can additionally take the supremum over $\beta>1$ to conclude the proof. (Note that all quantities in \eqref{dbpeq} are finite after taking $\sup_{\beta>1}$, except possible for the last term.  Therefore, the last term is also bounded after taking the supremum over $\beta>1$.)

\end{proof}

\section{Stable Mean Convex Sets are Convex}

\begin{theorem}\label{thm6}
Let $\Omega\subset\R^{\adimn}$.  Assume that $\Omega$ is mean convex (i.e. $H$ does not change sign on $\partial\Omega$).  Assume $\Omega\times\R$ minimizes Problem \ref{prob1}.  Then $\Omega$ or $\Omega^{c}$ is convex.
\end{theorem}
\begin{proof}
From Lemma \ref{varlem}, there exists $\scon\in\R$ such that
$$H(x)=\langle x,N(x)\rangle+\scon,\qquad\forall\,x\in\Sigma.$$
The case $\scon=0$ was already treated in Theorem \ref{thm3}.  We may therefore assume that $\scon\neq0$.  Replacing $\Omega$ with $\Omega^{c}$ (i.e. changing the direction of the unit exterior normal vector, which changes the signs of $H$ and $N$ and therefore of $\scon$), we may assume that $\scon<0$.  We may additionally assume that $H\geq0$ on $\Sigma$, since the case $H\leq0$ and $\scon<0$ was treated already in Theorem \ref{cor7}.

Since we assumed $H\geq0$ and $\scon<0$, we may freely use Lemma \ref{lemma39.2}.

Recall the definition of $\smax(A)$ in \eqref{smaxdef}.  Let $h\colonequals\max(\smax(A),0)$.  We may assume that there exists $\beta>0$ such that $h>0$ on a set of positive measure on $\Sigma$, otherwise all eigenvalues of $A$ are negative, i.e. $\Omega$ is convex, and the proof is complete.  From Lemma \ref{softeig},
\begin{equation}\label{six0h}
\begin{aligned}
&\int_{\Sigma}h Lh\gamma_{\sdimn}(x)\,\d x
\stackrel{\eqref{three4.3}\wedge\eqref{three4.5}}{=}\int_{\Sigma}\Big(h \mathcal{L}h+h^{2}(\vnorm{A}^{2}+1)\Big)\gamma_{\sdimn}(x)\,\d x\\
&=\int_{\Sigma}\Big[h^{2}(\vnorm{A}^{2}+1)+h\mathrm{Tr}\Big(\frac{e^{\beta A}}{\mathrm{Tr}(e^{\beta A})}\mathcal{L} A\Big)\\
&\qquad\qquad\qquad\qquad\qquad\qquad+\beta\sum_{i=1}^{\sdimn}\mathrm{Tr}\Big(\frac{e^{\beta A}}{\mathrm{Tr}(e^{\beta A})}\Big[\nabla_{e_{i}}A-I_{\sdimn}\mathrm{Tr}\Big(\frac{e^{\beta A}}{\mathrm{Tr}(e^{\beta A})}\nabla_{e_{i}}A\Big)\Big]^{2}\Big)\gamma_{\sdimn}(x)\,\d x\\
&\stackrel{\eqref{three4.5}}{=}\int_{\Sigma}\Big[h^{2}(\vnorm{A}^{2}+1)-h(\vnorm{A}^{2}+1)\mathrm{Tr}\Big(\frac{e^{\beta A}}{\mathrm{Tr}(e^{\beta A})}A\Big)+h\mathrm{Tr}\Big(\frac{e^{\beta A}}{\mathrm{Tr}(e^{\beta A})}LA\Big)\\
&\qquad\qquad\qquad\qquad\qquad\qquad+\beta\sum_{i=1}^{\sdimn}\mathrm{Tr}\Big(\frac{e^{\beta A}}{\mathrm{Tr}(e^{\beta A})}\Big[\nabla_{e_{i}}A-I_{\sdimn}\mathrm{Tr}\Big(\frac{e^{\beta A}}{\mathrm{Tr}(e^{\beta A})}\nabla_{e_{i}}A\Big)\Big]^{2}\Big)\gamma_{\sdimn}(x)\,\d x.
\end{aligned}
\end{equation}
(The quantity $\int_{\partial\Omega}h Lh\gamma_{\sdimn}(x)\,\d x$ is finite a priori by Lemma \ref{smaxlemma}.  And when we integrate by parts with $h$ in the second variation formula, that will be justified by Lemma \ref{smaxlemma} and Lemma \ref{lemma39.79}.)  The final term is nonnegative.  Using this and \eqref{three9p},
\begin{equation}\label{six1h}
\int_{\Sigma}h Lh\gamma_{\sdimn}(x)\,\d x
\geq \int_{\Sigma}\Big(h(\vnorm{A}^{2}+1)\Big[ h-\mathrm{Tr}\Big(\frac{A e^{\beta A}}{\mathrm{Tr}e^{\beta A}}\Big)\Big]+h\mathrm{Tr}\Big(\frac{e^{\beta A}}{\mathrm{Tr}e^{\beta A}}[2A-\lambda A^{2}]\Big)\gamma_{\sdimn}(x)\,\d x.
\end{equation}
We now let $\beta\to\infty$.  Then $h$ converges to the maximum of $0$ and the maximum eigenvalue $\alpha_{\mathrm{max}}(A)$ of $A$, as does $\mathrm{Tr}(\frac{A e^{\beta A}}{\mathrm{Tr}e^{\beta A}})$.  So, the first term vanishes by the Dominated Convergence Theorem and Lemma \ref{lemma39.2}.  Denote $g\colonequals\max(\alpha_{\mathrm{max}}(A),0)$.  Similarly, the Dominated Convergence Theorem gives
\begin{equation}\label{six2h}
\lim_{\beta\to\infty}\int_{\Sigma}h Lh\gamma_{\sdimn}(x)\,\d x
\geq \int_{\Sigma}g(2g-\scon g)\Big)\gamma_{\sdimn}(x)\,\d x.
\end{equation}
Recall that we assumed that $\scon<0$ and $h\geq0$, therefore $g\geq0$ and
$$
\lim_{\beta\to\infty}\int_{\Sigma}h Lh\gamma_{\sdimn}(x)\,\d x
>\int_{\Sigma}2g^{2}\gamma_{\sdimn}(x)\,\d x.
$$
(Equality cannot occur here since $g=0$ would imply that $\Omega$ is convex, and the proof would be complete.)  Since $\lim_{\beta\to\infty}\int_{\Sigma}h^{2}\gamma_{\sdimn}(x)\,\d x=\int_{\Sigma}g^{2}\gamma_{\sdimn}(x)\,\d x$, we have
\begin{equation}\label{six1r}
\lim_{\beta\to\infty}\frac{\int_{\Sigma}h Lh\gamma_{\sdimn}(x)\,\d x}{\int_{\Sigma}h^{2}\gamma_{\sdimn}(x)\,\d x}>2.
\end{equation}
we can now conclude the proof as in Theorem \ref{cor7}.

Let $f\colonequals h$ in Lemma \ref{orthlem}.  Note that $h(x)=h(-x)$ for all $x\in\Sigma$.  Then Lemma \ref{orthlem} and \eqref{six1r} imply that the variation of $\Omega\times\R$ correspond to $g$ satisfies
$$\frac{\d^{2}}{\d s^{2}}\Big|_{s=0}\int_{(\Sigma\times\R)^{(s)}}\gamma_{\sdimn}(x)dx<0.$$
This inequality violates the minimality of $\Omega$, achieving a contradiction, and completing the proof.
\end{proof}

\section{Symmetric Minimal Cones are Unstable}\label{seccone}

In this Section, we prove Case 3 of Theorem \ref{thm3}.  That is, we show that symmetric minimal cones are unstable for the Gaussian surface area, within the category of symmetric sets.  That is, a symmetric minimal cone can be perturbed to another nearby symmetric set with smaller Gaussian surface area, in a way that preserves the Gaussian volume of the set.

\begin{proof}[Proof of Case 3 of Theorem \ref{thm3}]
In this proof, in order to match the notation and formulas from \cite{zhu16}, we use a factor of $4$ in our Gaussian density, rather than a factor of $2$.

From the regularity part of Lemma \ref{lemma51}, since $H(x)=\langle x,N(x)\rangle=0$ for all $x\in\Sigma$ and $\Sigma$ is a cone (i.e. $\Sigma$ has a singularity at the origin), we must have $\sdimn\geq7$. Let $S^{\sdimn}\colonequals\{x\in\R^{\adimn}\colon\vnorm{x}=1\}$.  Let $W\colonequals\Sigma\cap S^{\sdimn}$ denote the cone $\Sigma$ intersected with the sphere, and let $M$ denote the regular part of $W$.  From the second variation formula, Lemma \ref{varlem2}, recall that
\begin{equation}\label{four8}
\frac{\d^{2}}{\d s^{2}}\Big|_{s=0}\int_{\Sigma^{(s)}}e^{-\vnorm{x}^{2}/4}(4\pi)^{-\sdimn/2}\,\d x
=-\int_{\Sigma}f(x)L'f(x)e^{-\vnorm{x}^{2}/4}(4\pi)^{-\sdimn/2}\,\d x,
\end{equation}
where, as opposed to \eqref{three4.5}, we now have \cite[Equation (1.12)]{zhu16},
\begin{equation}\label{threez}
L' \colonequals \Delta -\frac{1}{2}\langle x,\nabla \rangle+\frac{1}{2}+\vnormt{A}^{2}.
\end{equation}
For any $x\in\R^{\adimn}$, define $r=r(x)\colonequals\vnorm{x}$.  Using \cite[Equation (4.18)]{zhu16}, we can decompose the operator $L'$ from \eqref{threez} into its radial component and its spherical component
\begin{equation}\label{four7}
L' = r^{-2}(\widetilde{L}-(n-1)+L_{1}).
\end{equation}
Here $\widetilde{L}$ is the spherical component \cite[Equation (4.13)]{zhu16}
$$\widetilde{L}\colonequals \Delta_{M}+\vnormf{\widetilde{A}}^{2}+(\sdimn-1),$$
where $\widetilde{A}$ denotes the second fundamental form of $M$, and $L_{1}$ is the radial component
$$L_{1}\colonequals r^{2}\frac{\partial^{2}}{\partial r^{2}}+ (n-1)r\frac{\partial}{\partial r} -\frac{r^{3}}{2}\frac{\partial}{\partial r}+\frac{r^{2}}{2}.$$
We can then rewrite the quadratic form from \eqref{four8} as
\begin{equation}\label{four10}
\int_{\Sigma}f(x)L'f(x)\frac{e^{-\frac{\vnorm{x}^{2}}{4}}\,\d x}{(4\pi)^{\frac{n}{2}}}
=\int_{\Sigma\cap S^{n}}\int_{r=0}^{r=\infty}r^{n-3}f(r,\theta)(\widetilde{L}-(n-1)+L_{1})f(r,\theta) e^{-\frac{r^{2}}{4}}\frac{\d r \d\theta}{(4\pi)^{\frac{n}{2}}}.
\end{equation}
We first note that
\begin{equation}\label{four9}
L_{1}r=(n-1)r.
\end{equation}

Let $t(r)\colonequals r+\frac{4-2n}{r}$, $\forall$ $r>0$.  Then $t(r)$ is an eigenfunction of $L_{1}$, since
\begin{flalign*}
L_{1}t(r)
&\stackrel{\eqref{four9}}{=}(n-1)r + (4-2n)\Big(r^{2}2r^{-3}+(n-1)r(-r^{-2})-\frac{1}{2}r^{3}(-r^{-2})+\frac{r^{2}}{2}r^{-1}\Big)\\
&=(n-1)r + (4-2n)(2r^{-1}-(n-1)r^{-1}+r)\\
&=(-n+3)r + (4-2n)(3-n)r^{-1}
=-(n-3)f.
\end{flalign*}
Let $g(r)\colonequals r$, $\forall$ $r>0$.  Integrating by parts we get
\begin{equation}\label{four13}
\int_{0}^{\infty}r^{n-1}e^{-r^{2}/4}\,\d r=2(n-2)\int_{0}^{\infty}r^{n-3}e^{-r^{2}/4}\,\d r.
\end{equation}
We then obtain


\begin{equation}\label{four4}
\int_{0}^{\infty}t(r)g(r)r^{n-3}e^{-r^{2}/4}\,\d r=\int_{0}^{\infty}[r^{n-1}+(4-2n)r^{n-3}]e^{-r^{2}/4}\,\d r\stackrel{\eqref{four13}}{=}0.
\end{equation}
(Recall $\sdimn\geq7$ so this integral is finite a priori.)  And if $h(r)\colonequals  r+\frac{2-2n}{r}$, then
\begin{equation}\label{four2}
\int_{0}^{\infty} h(r)r^{n-1}e^{-r^{2}/4}\,\d r=\int_{0}^{\infty}[r^{n}+(2-2n)r^{n-2}]e^{-r^{2}/4}\,\d r=0.
\end{equation}
So, $h$ corresponds to a Gaussian volume-preserving perturbation.  Also, for all $r>0$,
\begin{equation}\label{four5}
\begin{aligned}
h(r)
&=\Big(r+\frac{4-2n}{r}\Big)\frac{2-2n}{4-2n}+r-\frac{2-2n}{4-2n}r\\
&=\Big(r+\frac{4-2n}{r}\Big)\frac{n-1}{n-2}+r\frac{2}{4-2n}
=\frac{n-1}{n-2}t(r)-\frac{1}{n-2}g(r).
\end{aligned}
\end{equation}
So, using $(L_{1}-(n-1))g\stackrel{\eqref{four9}}{=}0$, we have
\begin{equation}\label{four6}
\begin{aligned}
(L_{1}-(n-1))h(r)
&\stackrel{\eqref{four5}}{=}(L_{1}-(n-1))\frac{n-1}{n-2}t(r)
=\frac{n-1}{n-2}(-(n-3) - (n-1))t(r)\\
&=\frac{n-1}{n-2}(-2n+4)t(r)
=-2(n-1)t(r).
\end{aligned}
\end{equation}

Now, we will construct a function $f\colon\Sigma\to\R$ using a product of a radial function and a spherical function.  Define
\begin{equation}\label{fdef}
f(r,\theta)\colonequals\phi(\theta)h(r),\qquad\forall\,r>0,\,\,\forall\,\theta\in M.
\end{equation}
Starting with the radial integral in \eqref{four10}, we have
\begin{equation}\label{four11}
\begin{aligned}
&\int_{r=0}^{r=\infty}r^{n-3}f(\widetilde{L}-(n-1)+L_{1})f e^{-r^{2}/4}\,\d r \\
&\stackrel{\eqref{four6}}{=}\int_{r=0}^{r=\infty}r^{n-3}f\Big[\widetilde{L}f-2(n-1)t(r)\phi  \Big] e^{-r^{2}/4}\,\d r \\
&\stackrel{\eqref{four5}}{=}\int_{r=0}^{r=\infty}r^{n-3}\Big[f\widetilde{L}f-\Big(\frac{n-1}{n-2}t(r)-\frac{1}{n-2}g(r)\Big)(2(n-1)[t(r)]\phi^{2})  \Big] e^{-r^{2}/4}\,\d r \\
&\stackrel{\eqref{four4}}{=}\int_{r=0}^{r=\infty}r^{n-3}\Big[f\widetilde{L}f-\frac{n-1}{n-2}2(n-1)[t(r)]^{2}\phi^{2}  \Big] e^{-r^{2}/4}\,\d r \\
&\stackrel{\eqref{four4}\wedge\eqref{four5}}{=}\int_{r=0}^{r=\infty}r^{n-3}\Big[\phi\widetilde{L}\phi\Big[\Big(\frac{n-1}{n-2}\Big)^{2} [t(r)]^{2}+\frac{1}{(n-2)^{2}}[g(r)]^{2}\Big]\\
&\qquad\qquad\qquad\qquad\qquad\qquad\qquad\qquad\qquad\qquad-\frac{n-1}{n-2}2(n-1)[t(r)]^{2}\phi^{2}  \Big] e^{-r^{2}/4}\,\d r.
\end{aligned}
\end{equation}
(Recall $\sdimn\geq7$ so the above integrals are finite.)  (Even though $h$ is unbounded near the origin, we can use $h$ in the second variation formula by multiplying by cutoff functions, as in \cite[Lemma 6.6, Proposition 6.7]{zhu16}.)

We now use \cite[Theorem 0.3]{zhu16} (see also the proof of Theorem 8.3 in \cite{zhu16}) to find a nonnegative Dirichlet eigenfunction $\phi\colon M\to\R$ of $\widetilde{L}$ such that
\begin{equation}\label{leqn}
\widetilde{L}\phi=\kappa\phi,
\end{equation}
with $\kappa\geq 2(\sdimn-1)$, and such that $\phi(-\theta)=\phi(\theta)$ for all $\theta\in M$.  (Since $\Omega$ is symmetric, $\Sigma$ cannot be a hyperplane through the origin, i.e. $M$ is not totally geodesic, so \cite[Theorem 0.3]{zhu16} applies.  Also, since $\phi$ is nonnegative and $\Sigma$ is symmetric, $\phi(\cdot)+\phi(-\cdot)$ is a nonnegative eigenfunction of $\widetilde{L}$, i.e. we may assume a priori that $\phi$ itself is symmetric.)

Plugging \eqref{leqn} into \eqref{four11} and using the inequalities $[g(r)]^{2}>0$ and $\big(\frac{n-1}{n-2}\big)^{2}>\frac{n-1}{n-2}>1$,
\begin{equation}\label{four22}
\begin{aligned}
&\int_{r=0}^{r=\infty}r^{n-3}f(\widetilde{L}-(n-1)+L_{1})f e^{-r^{2}/4}\,\d r \\
&=\int_{r=0}^{r=\infty}r^{n-3}\Big[\kappa\phi^{2}\Big[\Big(\frac{n-1}{n-2}\Big)^{2} [t(r)]^{2}+\frac{1}{(n-2)^{2}}[g(r)]^{2}\Big]\\
&\qquad\qquad\qquad\qquad\qquad\qquad\qquad\qquad\qquad\qquad-\frac{n-1}{n-2}2(n-1)[t(r)]^{2}\phi^{2}  \Big] e^{-r^{2}/4}\,\d r \\
&>\int_{r=0}^{r=\infty}r^{n-3}\Big[\kappa\phi^{2} [t(r)]^{2}-2(n-1)[t(r)]^{2}\phi^{2}  \Big] e^{-r^{2}/4}\,\d r \\
&=\Big[\kappa\phi^{2}-2(n-1)\phi^{2}  \Big] \int_{r=0}^{r=\infty}r^{n-3}(t(r))^{2}e^{-r^{2}/4}\,\d r .
\end{aligned}
\end{equation}

Combining \eqref{four8}, \eqref{four10} and \eqref{four22}, then using $\kappa\geq2(n-1)$  \cite[Theorem 0.3]{zhu16},
\begin{equation}\label{four12}
\begin{aligned}
&\frac{\d^{2}}{\d s^{2}}\Big|_{s=0}\int_{\Sigma^{(s)}}e^{-\vnorm{x}^{2}/4}(4\pi)^{-\sdimn/2}\,\d x\\
&\quad<\int_{M}\Big[-\kappa[\phi(\theta)]^{2}+2(n-1)[\phi(\theta)]^{2}  \Big] \,\d\theta\cdot \int_{r=0}^{r=\infty}r^{n-3}(t(r))^{2}e^{-r^{2}/4}\,\d r\cdot (4\pi)^{-\sdimn/2}\leq0.
\end{aligned}
\end{equation}

Then \eqref{four12} implies that we have found a function $f$ such that $f(-x)=f(x)$ for all $x\in\Sigma$ with $\frac{\d^{2}}{\d s^{2}}|_{s=0}\int_{\Sigma^{(s)}}e^{-\vnorm{x}^{2}/4}(4\pi)^{-\sdimn/2}\,\d x<0$.  Moreover, from \eqref{four2}, \eqref{fdef} and Fubini's Theorem,
$$\int_{\Sigma}f(x) e^{-\vnorm{x}^{2}/4}\,\d x=0.$$
That is, we have shown that $\Sigma$ is unstable within the category of symmetric sets, i.e. $\Omega$ cannot minimize Problem \ref{prob1}.
\end{proof}

\medskip
\noindent\textbf{Acknowledgement}.  Thanks to Galyna Livshyts and Jonathan Zhu for helpful discussions.

\bibliographystyle{amsalpha}
%
\def\polhk#1{\setbox0=\hbox{#1}{\ooalign{\hidewidth
  \lower1.5ex\hbox{`}\hidewidth\crcr\unhbox0}}} \def\cprime{$'$}
  \def\cprime{$'$}
\providecommand{\bysame}{\leavevmode\hbox to3em{\hrulefill}\thinspace}
\providecommand{\MR}{\relax\ifhmode\unskip\space\fi MR }
\providecommand{\MRhref}[2]{%
  \href{http://www.ams.org/mathscinet-getitem?mr=#1}{#2}
}
\providecommand{\href}[2]{#2}

\end{document}